\documentclass[12pt]{amsart}
\usepackage{amsmath,amsfonts,amssymb,amscd,amsthm,amsbsy,epsf}
\newtheorem{thm}{Theorem}[section]

\newtheorem*{note}{Notation}

\newtheorem{cor}[thm]{Corollary}
\newtheorem{prop}[thm]{Proposition}
\theoremstyle{definition} 		%% bold "label" with roman text
\newtheorem{defn}[thm]{Definition}
\newtheorem{remark}[thm]{Remark}
\newtheorem{exa}[thm]{Example}
\def\nat{{\mathbb N}}
\def\zat{{\mathbb Z}}
\def\qat{{\mathbb Q}}

\begin{document}
\title{Topological dynamics indexed by words}
\author{Vassiliki  Farmaki and Andreas Koutsogiannis}

\begin{abstract}
Starting with a combinatorial partition theorem for words over an infinite alphabet dominated by a fixed sequence, established recently by the authors, we prove recurrence results for topological dynamical systems indexed by such words. In this way we extend the classical theory developed by Furstenberg and Weiss of dynamical systems indexed by the natural numbers to systems indexed by words. Moreover, applying this theory to topological systems indexed by semigroups that can be represented as words we get analogous recurrence results for such systems.
\end{abstract}


\maketitle
\baselineskip=18pt
\pagestyle{plain}

%%%%%%   BODY   %%%%%%%%%%%%%%%%%%%%%%%%%%%%%%

\section*{Introduction}
Furstenberg in collaboration with Weiss and Katznelson in the 1970's (\cite{Fu}, \cite{FuW}, \cite{FuKa}) connected fundamental combinatorial results, such as the partition theorems of van der Waerden (\cite{vdW}, 1927) and Hindman (\cite{H}, 1974), with topological dynamics and particularly with phenomena of (multiple) recurrence for suitable sequences of continuous functions defined on a compact metric space into itself.

The theorems of van der Waerden and Hindman were unified by a partition theorem for words over a finite alphabet of Carlson (\cite{C}, 1988); recently Carlson's theorem was essentially strengthened by the authors, in \cite{F}, \cite{FK}, to a partition theorem (Theorem~\ref{2}) for $\omega$-$\mathbb{Z}^\ast$-located words (i.e. words over an infinite alphabet dominated by a fixed sequence).

Our starting point in this work is a topological formulation of the partition theorem for $\omega$-$\mathbb{Z}^\ast$-located words (Theorem~\ref{3}). Introducing the notion of a dynamical system of continuous maps (homeomorphisms in the multiple case) from a compact metric space into itself indexed by $\omega$-$\mathbb{Z}^\ast$-located words, we apply this formulation to study (multiple) recurrence phenomena for these topological systems (Theorems~\ref{5}, ~\ref{114}), extending the earlier results of Birkhoff (\cite{Bi}) and Furstenberg-Weiss (\cite{Fu}, \cite{FuW}).

By making use of the representation of rational and integer numbers as $\omega$-$\mathbb{Z}^\ast$-located words (Example~\ref{exa1}) established by Budak-I\c{s}ik-Pym in \cite{BIP}, we obtain recurrence results for dynamical systems indexed by rational numbers or by the integers (Theorems~\ref{16}, ~\ref{17}, ~\ref{18}, ~\ref{19}). Moreover, we point out the way to obtain recurrence results for dynamical systems indexed by an arbitrary semigroup (Theorems~\ref{20}, ~\ref{22}).

\medskip

We will use the following notation.

\begin{note}
Let $\nat=\{1,2,\ldots\}$ be the set of natural numbers, $\mathbb{Z}=\{\ldots,-2,-1,0,1,2,\ldots\}$ the set of integer numbers, $\qat=\{\frac{m}{n}:m\in \zat,\;n\in \nat\}$ the set of rational numbers and $\zat^-=\{-n:n\in \nat\},$ $\zat^\ast=\zat\setminus\{0\},$ $\qat^{\ast}=\qat\setminus\{0\}.$
\end{note}

\section{A partition theorem for $\omega$-$\mathbb{Z}^\ast$-located words}

In this section we will introduce the $\omega$-$\mathbb{Z}^\ast$-located words and we will state a partition theorem for these words proved in \cite{FK}.

An {\bf{$\omega$-$\mathbb{Z}^\ast$-located word}} over the alphabet $\Sigma=\{\alpha_n: \;n\in \zat^\ast\}$ dominated by $\vec{k}=(k_n)_{n\in \zat^\ast},$ where $k_n\in \nat$ for every $n\in \zat^\ast$ and
$(k_n)_{n\in\nat}$, $(k_{-n})_{n\in\nat}$ are increasing sequences, is a function $w$ from a non-empty,  finite subset $F$ of $\mathbb{Z}^\ast$ into the alphabet $\Sigma$ such that $w(n)=w_n\in \{\alpha_1, \ldots, \alpha_{k_n}\}$ for every $n\in F\cap \nat$ and $w_n\in \{\alpha_{-k_{n}}, \ldots,\alpha_{-1}\}$ for every $n\in F\cap \mathbb{Z}^{-}$. So, the set
$\widetilde{L}(\Sigma, \vec{k})$ of all (constant) $\omega$-$\mathbb{Z}^\ast$-located words over $\Sigma$ dominated by $\vec{k}$ is:
\begin{center}$\widetilde{L}(\Sigma, \vec{k})= \{w=w_{n_1}\ldots w_{n_l} : l\in\nat, n_1<\ldots<n_l\in \mathbb{Z}^\ast$ and $ w_{n_i}\in
\{\alpha_1, \ldots, \alpha_{k_{n_i}}\}$ if $\;\;\;$\\
 $n_i> 0$, $ w_{n_i}\in \{\alpha_{-k_{n_i}}, \ldots,\alpha_{-1}\}$ if $n_i< 0$
for every $1\leq i \leq l \}.$
\end{center}
Analogously, the set of {\bf{$\omega$-located words}} over the alphabet $\Sigma=\{\alpha_n: \;n\in \nat\}$ dominated by the increasing sequence $\vec{k}=(k_n)_{n\in \nat}\subseteq \nat$ is
\begin{center}$L(\Sigma, \vec{k})= \{w=w_{n_1}\ldots w_{n_l} : l\in\nat, n_1<\ldots<n_l\in \nat$ and $ w_{n_i}\in
\{\alpha_1, \ldots, \alpha_{k_{n_i}}\}\;\;\;\;\;\;\;\;$\\ for every $1\leq i \leq l \}.\;\;\;\;\;\;\;\;\;\;\;\;\;\;\;\;\;\;\;\;\;\;\;\; \;\;\;\;\;\;\;\;\;\;\;\;\;\;\;\;\;\;\;\;\;\;\;\; \;\;\;\;\;\;\;\;\;\;\;\;\;$
\end{center}
\begin{exa}\label{exa1}
We will give some examples of sets that can be represented as $\omega$-$\mathbb{Z}^\ast$-located words.

(1) According to Budak-I\c{s}ik-Pym in \cite{BIP}, every rational number $q$  has a unique expression in the form $$q=\sum^{\infty}_{s=1}q_{-s}\frac{(-1)^{s}}{(s+1)!}\;+\;\sum^{\infty}_{r=1}q_{r}(-1)^{r+1}r! $$
where $(q_n)_{n \in \mathbb{Z}^\ast}\subseteq \nat\cup\{0\}$ with $\;0\leq q_{-s}\leq s$ for every $s>0$, $ 0\leq q_r\leq r$ for every $r> 0$ and $q_{-s}=q_r=0$ for all but finite many $r,s$. Setting $\Sigma=\{\alpha_n:\;n\in \mathbb{Z}^\ast\}$, where  $\alpha_{-n}=\alpha_n=n$ for $n\in \nat,$ and $\vec{k}=(k_n)_{n\in\mathbb{Z}^\ast}$, where $k_{-n}=k_n=n$ for $n\in \nat,$ the function
\begin{center}
$g^{-1}: \mathbb{Q}^{\ast}\rightarrow\widetilde{L}(\Sigma,\vec{k}),$
\end{center}
 which sends $q$ to the word  $w=q_{t_1}\ldots q_{t_l}\in \widetilde{L}(\Sigma,\vec{k})$, where $\{t_1,\ldots,t_l\}=\{t\in\mathbb{Z}^\ast : q_t \neq 0\},$ is one-to-one and onto.

(2) According to \cite{BIP}, for a given increasing sequence $(k_n)_{n\in \nat}\subseteq \nat$ with $k_n \geq 2$, every integer number $z\in \zat$ has a unique expression in the form $$z=\sum^{\infty}_{s=1}z_s(-1)^{s-1}l_{s-1}$$ where $l_0=1,$
$l_s=k_1\ldots k_s ,$ for $s\in\nat$ and $(z_s)_{s\in \nat}\subseteq \nat\cup\{0\}$ with $0\leq z_s\leq k_s$ for every $s\in\nat$ and $z_s=0$ for all but finite many $s.$ Setting $\Sigma=\{\alpha_n:\;n\in \nat\}$, where  $\alpha_n=n$, and $\vec{k}=(k_n)_{n\in\nat}$ the function
\begin{center}
$g^{-1}: \zat^{\ast}\rightarrow L(\Sigma,\vec{k}),$
\end{center}
which sends $z$ to the word  $w=z_{s_1}\ldots z_{s_t}\in L(\Sigma,\vec{k})$, where $\{s_1,\ldots,s_t\}=\{s\in\nat : z_s \neq 0\},$ is one-to-one and onto.

(3) For a given natural number $k>1,$ every natural number $n$ has a unique expression in the form $$n=\sum^{\infty}_{s=1}n_s k^{s-1}$$ where $(n_s)_{s\in \nat}\subseteq \nat\cup\{0\}$ with $0\leq n_s\leq k-1$ and $n_s=0$ for all but finite many $s.$ Setting $\Sigma=\{1,\ldots,k-1\}$   and $\vec{k}=(k_n)_{n\in\nat}$ with $k_n=k-1$ the function
\begin{center}
$g^{-1}: \nat\rightarrow L(\Sigma,\vec{k}),$
\end{center}
which sends $n$ to the word  $w=n_{s_1}\ldots n_{s_l}\in L(\Sigma,\vec{k}),$ where ${\{s_1,\ldots,s_l\}=\{s\in\nat : n_s \neq 0\},}$ is one-to-one and onto.
\end{exa}

Let $\Sigma=\{\alpha_n:\;n\in\mathbb{Z}^\ast\}$ be an alphabet,  $\vec{k}=(k_n)_{n\in\mathbb{Z}^\ast}\subseteq\nat$ such that
$(k_n)_{n\in\nat}$, $(k_{-n})_{n\in\nat}$ are increasing sequences and  $\upsilon \notin \Sigma$ be an entity which is called a {\bf{variable}}.

The set of {\bf{variable  $\omega$-$\mathbb{Z}^\ast$-located words}} over $\Sigma$ dominated by $\vec{k}$ is:
\begin{center}$\widetilde{L}(\Sigma, \vec{k} ; \upsilon) = \{w=w_{n_1}\ldots w_{n_l} : l\in\nat, n_1<\ldots<n_l\in \mathbb{Z}^\ast,$ $ w_{n_i}\in \{\upsilon, \alpha_1, \ldots, \alpha_{k_{n_i}} \}$ if $\;\;\;$\\
 $\;\;\;\;\;\;\;\;\;\;\;\;\;\;\;\;\;\;n_i> 0$, $ w_{n_i}\in \{\upsilon,\alpha_{-k_{n_i}}, \ldots,\alpha_{-1}\}$ if $n_i< 0$ for all $1\leq i \leq l $ and there\\ exists $1\leq i \leq l$  with $w_{n_i}=\upsilon \}.\;\;\;\;\;\;\;\;\;\;\;\;\;\;\;\;\;\; \;\;\;\;\;\;\;\;\;\;\;\;\;\;\;\;\;\; \;\;\;\;$\end{center}
The set of {\bf{variable  $\omega$-located words}} over $\Sigma=\{\alpha_n: \;n\in \nat\}$ dominated by the increasing sequence $\vec{k}=(k_n)_{n\in \nat}\subseteq \nat$ is:
\begin{center}$L(\Sigma, \vec{k} ; \upsilon) = \{w=w_{n_1}\ldots w_{n_l} : l\in\nat,\; n_1<\ldots<n_l\in \nat,$ $ w_{n_i}\in \{\upsilon, \alpha_1, \ldots, \alpha_{k_{n_i}} \}\;\;\;\;\;\;\;\;\;$ \\ $\;$ for all $1\leq i \leq l $ and there exists $1\leq i \leq l$  with $w_{n_i}=\upsilon \}\;\;\;$\end{center}
We set $\widetilde{L}(\Sigma\cup\{\upsilon\}, \vec{k})=\widetilde{L}(\Sigma, \vec{k})\cup \widetilde{L}(\Sigma, \vec{k} ; \upsilon)$ and $L(\Sigma\cup\{\upsilon\}, \vec{k})=L(\Sigma, \vec{k})\cup L(\Sigma, \vec{k} ; \upsilon)$.

For $w=w_{n_1}\ldots w_{n_l}\in \widetilde{L}(\Sigma\cup\{\upsilon\}, \vec{k})$ the set $dom(w)=\{n_1,\ldots ,n_l\}$ is the \textit{domain} of $w$. Let $dom^-(w)=\{n\in dom(w):\;n<0\}$ and $dom^+(w)=\{n\in dom(w):\;n>0\}.$
\newline
We define the set
\begin{center}
$\;\;\widetilde{L}_{0}(\Sigma,\vec{k};\upsilon)=\{w\in \widetilde{L}(\Sigma,\vec{k};\upsilon): w_{i_1}=\upsilon=w_{i_2}$ for some $i_1\in dom^-(w),i_2\in dom^+(w)\}$.
\end{center}
For $w=w_{n_1}\ldots w_{n_r},  u=u_{m_1}\ldots u_{m_l}\in \widetilde{L}(\Sigma \cup\{\upsilon\}, \vec{k})$ with $dom(w)\cap dom(u)=\emptyset$ we  define the \textbf{concatenating word}:
\begin{center}
$w \star u =z_{q_1}\ldots z_{q_{r+l}}\in \widetilde{L}(\Sigma \cup\{\upsilon\}, \vec{k}),$
\end{center}
where $\{q_1<\ldots<q_{r+l}\}=dom(w)\cup dom(u),\;z_i=w_i$ if $i \in dom(w)$ and $z_i=u_i$ if $i \in dom(u)$.

The set $\widetilde{L}(\Sigma\cup\{\upsilon\}, \vec{k})$ can be endowed with \textbf{the relations} $<_{\textsl{R}_1},$ $<_{\textsl{R}_2}:$
\begin{center}
$w <_{\textsl{R}_1} u \Longleftrightarrow \;dom(u)=A_1\cup A_2$ with $A_1,A_2\neq \emptyset$ such that \\ $\;\;\;\;\;\;\;\;\;\;\;\;\;\;\;\;\;\;\;\;\;\;\;\; \;\;\;\;\max A_1<\min dom(w)\leq \max dom(w)<\min A_2,$
\end{center}
\begin{center}
$w <_{\textsl{R}_2} u \Longleftrightarrow \;\max dom(w)< \min dom(u). \;\;\;\;\;\;\;\;\;\;\;\;\;\;\;\;\;\;\;\;\;\;\;\;\;\;\;$
\end{center}
We define the sets

 $\widetilde{L}^\infty (\Sigma, \vec{k} ; \upsilon) = \{\vec{w} = (w_n)_{n\in\nat} : w_n\in \widetilde{L}_0(\Sigma, \vec{k} ; \upsilon)$
and $w_n<_{\textsl{R}_1}w_{n+1}$ for every $ n\in\nat\},$

$L^\infty (\Sigma, \vec{k} ; \upsilon) = \{\vec{w} = (w_n)_{n\in\nat} : w_n\in L(\Sigma, \vec{k} ; \upsilon)$
and  $w_n<_{\textsl{R}_2}w_{n+1}$ for every $ n\in\nat\}.$

We will define now the notion of \textbf{substitution} for the variable  $\omega$-$\mathbb{Z}^\ast$-located words and respectively for the variable  $\omega$-located words.

Let $w=w_{n_1}\ldots w_{n_l}\in \widetilde{L}_0(\Sigma, \vec{k} ; \upsilon)$ with $n_w=\min dom^+(w)$ and $-m_w=\max dom^-(w)$ for $n_w, m_w\in\nat$. For every $(p,q)\in \{1,\ldots,k_{n_w}\}\times \{1,\ldots,k_{-m_w}\}\cup\{(\upsilon,\upsilon)\}$ we set:
\begin{center}
$w(\upsilon,\upsilon)=w$ and $w(p,q)=u_{n_1}\ldots u_{n_l}$,
\end{center}
for every $(p,q)\in \{1,\ldots,k_{n_w}\}\times \{1,\ldots,k_{-m_w}\}$, where, for $1\leq i\leq l$, $u_{n_i}=w_{n_i}$ if $w_{n_i}\in \Sigma$, $u_{n_i}=\alpha_p$ if $w_{n_i}=\upsilon,$ $n_i> 0$ and $u_{n_i}=\alpha_{-q}$ if $w_{n_i}=\upsilon,$ $n_i< 0$.

Respectively, Let $w=w_{n_1}\ldots w_{n_l}\in L(\Sigma, \vec{k} ; \upsilon)$ with $n_w=\min dom(w)\in\nat$. For every $p\in \{1,\ldots,k_{n_w}\}\cup\{\upsilon\}$ we set:
\begin{center}
$w(\upsilon)=w$ and $w(p)=u_{n_1}\ldots u_{n_l}$,
\end{center}
for every $p\in \{1,\ldots,k_{n_w}\}$, where, for $1\leq i\leq l$, $u_{n_i}=w_{n_i}$ if $w_{n_i}\in \Sigma$, $u_{n_i}=\alpha_p$ if $w_{n_i}=\upsilon$.
\newline
We remark that for $\vec{w} =(w_n)_{n\in\nat}\in \widetilde{L}^\infty (\Sigma, \vec{k} ; \upsilon)$ (resp. for $\vec{w} =(w_n)_{n\in\nat}\in L^\infty (\Sigma, \vec{k} ; \upsilon)$) we have $n\leq\min dom^+(w_n)$ and $-n\geq\max dom^-(w_n)$ (resp. $n\leq\min dom(w_n)$), for $n\in \nat$. So, for $n\in \nat$, the substituted word $w_n(p,q)$ (resp. $w_n(p)$ ) has meaning for every  $(p,q)\in\nat\times \nat$ with $p\leq k_n$ and $q\leq k_{-n}$ (resp. for every $p\in\nat$ with $p\leq k_n$ ).

Fix a sequence $\vec{w} = (w_n)_{n\in\nat}\in \widetilde{L}^\infty (\Sigma, \vec{k} ; \upsilon)$ (resp. $\vec{w} = (w_n)_{n\in\nat}\in L^\infty (\Sigma, \vec{k} ; \upsilon)$).

An\textbf{ extracted $\omega$-$\mathbb{Z}^\ast$-located word} (resp. \textbf{extracted $\omega$-located word}) of $\vec{w}$ is an $\omega$-$\mathbb{Z}^\ast$-located word $z\in \widetilde{L}(\Sigma, \vec{k})$ (resp. $z\in L(\Sigma, \vec{k})$) with
\begin{center}
$z=w_{n_1}(p_1,q_1)\star \ldots \star w_{n_\lambda}(p_\lambda,q_\lambda)$ (resp. $z=w_{n_1}(p_1)\star \ldots \star w_{n_\lambda}(p_\lambda)$),
\end{center}
where $\lambda\in\nat$, $n_1<\ldots<n_\lambda\in\nat$ and $(p_i,q_i) \in \{1,\ldots,k_{n_i}\}\times \{1,\ldots,k_{-n_i}\} $ (resp. $p_i \in \{1,\ldots,k_{n_i}\}$) for every $1\leq i \leq \lambda$.
The set of all the extracted $\omega$-$\mathbb{Z}^\ast$-located words of $\vec{w}$ is denoted by $\widetilde{E}(\vec{w})$ (resp. all the extracted $\omega$-located words of $\vec{w}$ is denoted by $E(\vec{w})$).

An \textbf{extracted variable $\omega$-$\mathbb{Z}^\ast$-located word} (resp. \textbf{extracted variable $\omega$-located word}) of $\vec{w}$ is a variable $\omega$-$\mathbb{Z}^\ast$-located word $u\in \widetilde{L}_0(\Sigma, \vec{k} ; \upsilon)$ (resp. $u\in L(\Sigma, \vec{k} ; \upsilon)$) with
\begin{center}
$u=w_{n_1}(p_1,q_1)\star \ldots \star w_{n_\lambda}(p_\lambda,q_\lambda)$ (resp. $u=w_{n_1}(p_1)\star \ldots \star w_{n_\lambda}(p_\lambda)$),
\end{center}
where $\lambda\in\nat$, $n_1<\ldots<n_\lambda\in\nat$, $(p_i,q_i) \in \{1,\ldots,k_{n_i}\}\times \{1,\ldots,k_{-n_i}\}\cup\{(\upsilon,\upsilon)\} $  for every $1\leq i \leq \lambda$ and $(\upsilon,\upsilon) \in \{ (p_1,q_1),\ldots,(p_\lambda,q_\lambda)\}$ (resp. $p_i\in \{1,\ldots,k_{n_i}\}\cup\{\upsilon\} $ for every $1\leq i \leq \lambda$ and $\upsilon \in \{ p_1,\ldots,p_\lambda\}$  ).
The set of all the extracted variable $\omega$-$\mathbb{Z}^\ast$-located words of $\vec{w}$ is denoted by $\widetilde{EV}(\vec{w})$ (resp. the set of all the extracted variable $\omega$-located words of $\vec{w}$ is denoted by $EV(\vec{w})$). Let
\begin{center}
$\widetilde{EV}^{\infty}(\vec{w}) = \{\vec{u}=(u_n)_{n\in\nat} \in \widetilde{L}^\infty (\Sigma, \vec{k} ; \upsilon) : u_n\in \widetilde{EV}(\vec{w})$ for every $n\in\nat \},$
\end{center}
\begin{center}
 $EV^{\infty}(\vec{w}) = \{\vec{u}=(u_n)_{n\in\nat} \in L^\infty (\Sigma, \vec{k} ; \upsilon) : u_n\in EV(\vec{w})$ for every $n\in\nat \}.$
\end{center}
If $\vec{u} \in \widetilde{EV}^{\infty}(\vec{w})$ (resp. $\vec{u} \in EV^{\infty}(\vec{w})$), then we say that $\vec{u}$ is an \textbf{extraction }of $\vec{w}$ and we write $\vec{u} \prec \vec{w}$. Notice that for $\vec{u},\vec{w}\in \widetilde{L}^{\infty} (\Sigma, \vec{k} ; \upsilon)$ (resp. $\vec{u},\vec{w}\in L^{\infty} (\Sigma, \vec{k} ; \upsilon)$) we have $\vec{u} \prec \vec{w}$ if and only if $\widetilde{EV}(\vec{u})\subseteq \widetilde{EV}(\vec{w})$ (resp. $EV(\vec{u})\subseteq EV(\vec{w})$).

\medskip

Using the theory of ultrafilters we proved in \cite{F}, \cite{FK} the following partition theorem for $\omega$-$\mathbb{Z}^\ast$-located words and for $\omega$-located words.

\begin{thm}
\label{2}(\cite{F}, \cite{FK})
Let $\Sigma=\{\alpha_n \;:\;n \in \mathbb{Z}^{\ast}\}$ be an alphabet,  $\vec{k}=(k_n)_{n\in\mathbb{Z}^\ast}\subseteq \nat$ such that
$(k_n)_{n\in\nat}$, $(k_{-n})_{n\in\nat}$ are increasing sequences, $\upsilon \notin \Sigma$ and let $\vec{w}=(w_n)_{n\in \nat}\in \widetilde{L}^{\infty} (\Sigma, \vec{k} ; \upsilon)$ (resp. $\vec{w}=(w_n)_{n\in \nat}\in L^{\infty} (\Sigma, \vec{k} ; \upsilon)$). If $\widetilde{L}(\Sigma, \vec{k})=C_1\cup\ldots\cup C_s$ (resp. $L(\Sigma, \vec{k})=C_1\cup\ldots\cup C_s$), $s\in \nat,$ then there exists $\vec{u}\prec \vec{w}$ and $1\leq j_0 \leq s$ such that
\begin{center}
 $\widetilde{E}(\vec{u})\subseteq C_{j_0}$ (resp. $E(\vec{u})\subseteq C_{j_0}$).
\end{center}
\end{thm}

\section{Implications of the partition theorem to topological dynamics  }

We will prove a topological formulation (in Theorem~\ref{3}) of the partition Theorem~\ref{2}, important for proving later (multiple) recurrence results for  systems of continuous maps from a compact metric space into itself indexed by $\omega$-$\mathbb{Z}^\ast$-located words (Theorem~\ref{5}), which extend fundamental recurrence results of Birkhoff (\cite{Bi}) and Furstenberg-Weiss (\cite{Fu}, \cite{FuW}).

Let an alphabet $\Sigma=\{\alpha_n: \;n\in \zat^\ast\}$ and   $\vec{k}=(k_n)_{n\in \zat^\ast}\subseteq \nat,$ where
$(k_n)_{n\in\nat}$, $(k_{-n})_{n\in\nat}$ are increasing sequences. Observe that $\widetilde{L}(\Sigma, \vec{k}) $ can be considered as a directed set with partial order either $ \textsl{R}_1$ or $ \textsl{R}_2$. So, in a topological space $X$, we can consider  $\{x_w\}_{w\in \widetilde{L} (\Sigma, \vec{k})}\subseteq X$ either as an $ \textsl{R}_1$-net or as an $ \textsl{R}_2$-net in $X$. Consequently,
$\{x_w\}_{w\in L (\Sigma, \vec{k})}$ is an $ \textsl{R}_2$-subnet of $\{x_w\}_{w\in \widetilde{L} (\Sigma, \vec{k})}$. Moreover,  $\{x_w\}_{w\in \widetilde{E}(\vec{u})}$ for $\vec{u}\in \widetilde{L}^{\infty} (\Sigma, \vec{k};\upsilon)$ is an $ \textsl{R}_1$-subnet of $\{x_w\}_{w\in \widetilde{L} (\Sigma, \vec{k})}$ and respectively $\{x_w\}_{w\in E(\vec{u})}$ for $\vec{u}\in L^{\infty} (\Sigma, \vec{k};\upsilon)$ is an $ \textsl{R}_2$-subnet of $\{x_w\}_{w\in L (\Sigma, \vec{k})}$.
\newline
Let $x_0\in X.$ We write
$$\text{$R_1$-}\lim_{w\in \widetilde{L} (\Sigma, \vec{k})}x_w=x_0$$
if $\{x_w\}_{w\in \widetilde{L} (\Sigma, \vec{k})}$ converges to $x_0$ as $ \textsl{R}_1$-net in $X$, i.e. if for any neighborhood $V$ of $x_0,$ there exists $n_0\equiv n_0(V)\in \nat$ such that $x_w\in V$ for every $w$ with $\min\{-\max dom^-(w),$ $\min dom^+(w)\}\geq n_0$. Analogously, we write
$$\text{$R_2$-}\lim_{w\in L(\Sigma, \vec{k})}x_w=x_0$$
if for any neighborhood $V$ of $x_0,$ there exists $n_0\equiv n_0(V)\in \nat$ such that $x_w\in V$ for every $w$ with $\min dom(w)\geq n_0$.

We will give now a topological reformulation of Theorem~\ref{2}.

\begin{thm}\label{3}
Let $(X,d)$ be a compact metric space, $\Sigma=\{\alpha_n \;:\;n \in \mathbb{Z}^{\ast}\}$ be an alphabet,  $\vec{k}=(k_n)_{n\in\mathbb{Z}^\ast}\subseteq \nat$ such that
$(k_n)_{n\in\nat}$, $(k_{-n})_{n\in\nat}$ are increasing sequences, $\upsilon \notin \Sigma$ and $\vec{w}=(w_n)_{n\in \nat}\in \widetilde{L}^{\infty} (\Sigma, \vec{k} ; \upsilon)$ (resp. $\vec{w}=(w_n)_{n\in \nat}\in L^{\infty} (\Sigma, \vec{k} ; \upsilon)$). For every net $\{x_ w\}_{ w\in\widetilde{L} (\Sigma, \vec{k})}\subseteq X$ (resp. $\{x_ w\}_{ w\in L(\Sigma, \vec{k})}\subseteq X$),   there exist an extraction $\vec{u}\prec\vec{w}$ of $\vec{w}$ and $x_0\in X$ such that $$\text{$R_1$-}\lim_{w\in \widetilde{E}(\vec{u})}x_w=x_0\;\;(\text{resp.\;\;$R_2$-}\lim_{w\in E(\vec{u})}x_w=x_0).$$
\end{thm}

\begin{proof} For $x\in X$ and $\epsilon>0$ we set $\widehat{B}(x,\epsilon)=\{y\in X : d(x,y)\leq \epsilon\}$.
Since $(X,d)$ is a compact metric space, we have that $X=\bigcup_{i=1}^{m_1}\widehat{B}(x_i^1,\frac{1}{2})$ for some $x_1^1,\ldots,x^1_{m_1}\in X$. According to Theorem~\ref{2}, there exists $\vec{u}_1\prec\vec{w}$ and $1\leq i_1\leq m_1$ such that
$\{x_w\}_{w\in \widetilde{E}(\vec{u}_1)}\subseteq \widehat{B}(x_{i_1}^1,\frac{1}{2})$ (resp. $\{x_w\}_{w\in E(\vec{u}_1)}\subseteq \widehat{B}(x_{i_1}^1,\frac{1}{2})$ ). Analogously, since $\widehat{B}(x_{i_1}^1,\frac{1}{2})$ is compact,  there exist $x_1^2,\ldots,x^2_{m_2}\in X,$ such that $\widehat{B}(x_{i_1}^1,\frac{1}{2})\subseteq \bigcup_{i=1}^{m_2}\widehat{B}(x_i^2,\frac{1}{4}),$ and consequently there exist $\vec{u}_2\prec\vec{u}_1$ and $1\leq i_2\leq m_2$ such that $\{x_w\}_{w\in \widetilde{E}(\vec{u}_2)}\subseteq\widehat{B}(x_{i_1}^1,\frac{1}{2})\cap\widehat{B}(x_{i_2}^2,\frac{1}{4})$. Inductively, we construct $(\vec{u}_n)_{n\in \nat}\subseteq \widetilde{L}^\infty (\Sigma, \vec{k} ; \upsilon)$ (resp. $(\vec{u}_n)_{n\in \nat}\subseteq L^\infty (\Sigma, \vec{k} ; \upsilon)$) such that $\vec{u}_{n+1}\prec\vec{u}_n\prec \vec{w}$ for every $n\in \nat$ and
closed balls $\widehat{B}(x_{i_{n}}^{n},\frac{1}{2^{n}}),$ for $n\in \nat$ such that for every $n\in \nat$
\begin{center}
$\{x_w\}_{w\in \widetilde{E}(\vec{u}_n)}\subseteq  \bigcap_{j=1}^n\widehat{B}(x_{i_j}^j,\frac{1}{2^j})$ (resp. $\{x_w\}_{w\in E(\vec{u}_n)}\subseteq  \bigcap_{j=1}^n\widehat{B}(x_{i_j}^j,\frac{1}{2^j}))$.
\end{center}
If  $\vec{u}_n=(w^{(n)}_k)_{k\in \nat}$ for every $n\in \nat,$ then we set $\vec{u}=(w_n^{(n)})_{n\in \nat}.$ Of course $\vec{u}\prec\vec{w}$.
Let $\{x_0\}=\bigcap_{n\in \nat}\widehat{B}(x_{i_{n}}^{n},\frac{1}{2^{n}}).$
Then $R_1$-$\lim_{ w\in \widetilde{E}(\vec{u})}x_w=x_0$ (resp. $R_2$-$\lim_{ w\in E(\vec{u})}x_w=x_0$).  Indeed, for $\varepsilon>0$ pick $k_0\in \nat$ such that $1/2^{k_0}<\varepsilon.$ Then, for every $ w\in \widetilde{E}(\vec{u}_{k_0})$ we have that $d(x_{w},x_0)\leq 1/2^{k_0}<\varepsilon.$ Since $\widetilde{E}(\vec{u}_n)\subseteq \widetilde{E}(\vec{u}_{k_0})$ for every $n\geq k_0,$ we have that $\widetilde{E}((w_n^{(n)})_{n\geq k_0})\subseteq \widetilde{E}(\vec{u}_{k_0})$ and consequently that $\{w\in \widetilde{E}(\vec{u}):\min\{-\max dom^-(w),$ $\min dom^+(w)\}\geq n_0\}\subseteq \widetilde{E}(\vec{u}_{k_0})$ for  $n_0=\max\{-\min dom^-(w_{k_0}^{(k_0)}),\max dom^+(w_{k_0}^{(k_0)})\}$.
\end{proof}

\begin{remark}\label{rem2}
 (1)  Note that Theorem~\ref{3} follows from Theorem~\ref{2}. But conversely Theorem~\ref{2} follows from Theorem~\ref{3}. In fact, one only needs the assertion for finite spaces. Indeed, let $\widetilde{L}(\Sigma, \vec{k})=C_1\cup\ldots\cup C_s$ (resp. $L(\Sigma, \vec{k})=C_1\cup\ldots\cup C_s$), $s\in \nat.$ Then defining, for every $w\in \widetilde{L}(\Sigma, \vec{k})$ (resp. for $w\in L(\Sigma, \vec{k})$), $x_w =i$ if and only if $w\in C_i$ and $w\notin C_j$ for all $j<i$, we have, according to Theorem~\ref{3}, that there exist $\vec{u}=(u_n)_{n\in\nat}\prec\vec{w}$ and $1\leq j_0\leq s$ such that $\text{$R_1$-}\lim_{w\in \widetilde{E}(\vec{u})}x_w=j_0$ (resp. $\text{ $R_2$-}\lim_{w\in E(\vec{u})}x_w=j_0).$ For $n_0$ large enough and $\vec{u}_0=(u_{n+n_0})_{n\in\nat}$ we have that $\widetilde{E}(\vec{u}_0)\subseteq C_{j_0}$ (resp. $E(\vec{u}_0)\subseteq C_{j_0}$).

 (2) Observe that if $\text{$R_1$-}\lim_{w\in \widetilde{E}(\vec{u})}x_w=x_0$ for $\vec{u}=(u_n)_{n\in \nat}\in \widetilde{L}^{\infty} (\Sigma, \vec{k} ; \upsilon)$, then the sequences $(x_{u_n (p_n,q_n)})_{n\in\nat}$ converge uniformly to $x_0$ for all the sequences $((p_n,q_n))_{n\in\nat}\subseteq \nat\times \nat$ with $1\leq p_n\leq k_{n}$, $1\leq q_n \leq k_{-n}$. Analogously, if $\text{$R_2$-}\lim_{w\in E(\vec{u})}x_w=x_0$ for $\vec{u}=(u_n)_{n\in \nat}\in L^{\infty} (\Sigma, \vec{k} ; \upsilon)$, then the sequences $(x_{u_n (p_n)})_{n\in\nat}$ converge uniformly to $x_0$ for all the sequences $(p_n)_{n\in\nat}\subseteq \nat$ with $1\leq p_n\leq k_{n}$.

 (3) The particular case of Theorem~\ref{2} for words in $L(\Sigma, \vec{k}),$ where $\Sigma$ is a finite alphabet, gives Carlson's partition theorem in \cite{C}, whose topological reformulation has been given by Furstenberg and Katznelson in \cite{FuKa}.

 (4) The particular case of Theorem~\ref{2} for words in $L(\Sigma, \vec{k})$ where $\Sigma$ is a singleton and $\vec{k}=(k_n)_{n\in\nat}$ with $k_n =1$ for all $n\in\nat$ (so, the words can be coincide with its domain) is Hindman's partition theorem in \cite{H}. Furstenberg and Weiss in \cite{FuW}
gave the topological reformulation of Hindman's theorem introducing the $IP$-convergence of a net $\{x_F\}_{F\in [\nat]^{<\omega}_{>0}}$ in a topological space $X$ to $x_0\in X$, i.e. if for any neighborhood $V$ of $x_0,$ there exists $n_0\equiv n_0(V)\in \nat$ such that $x_F\in V$ for every $F\in [\nat]^{<\omega}_{>0}$ with $\min F\geq n_0$. In this case we write IP-$\lim_{F\in [\nat]^{<\omega}_{>0}}x_F=x_0$. Also, using the $IP$-convergence, they proved important results in topological dynamics (see \cite{Fu}).
\end{remark}

In the following proposition we will characterize the $R_1$-convergence of nets $\{x_w\}_{w\in \widetilde{L} (\Sigma, \vec{k})}$ and the $R_2$-convergence of nets $\{x_w\}_{w\in L(\Sigma, \vec{k})}$ as uniform $IP$-convergence, pointing
out the way for strengthening results involving the $IP$-convergence.

\begin{prop}\label{4}
Let $X$ be a topological space, $\vec{w}=(w_n)_{n\in \nat}\in  \widetilde{L}^{\infty}(\Sigma,\vec{k};\upsilon)$ $($resp. $\vec{w}=(w_n)_{n\in \nat}\in  L^{\infty}(\Sigma,\vec{k};\upsilon))$ and $\{x_ w\}_{ w\in\widetilde{L} (\Sigma, \vec{k})}\subseteq X$ (resp. $\{x_ w\}_{ w\in L(\Sigma, \vec{k})}\subseteq X$). For a sequence $((p_n,q_n))_{n\in\nat}\subseteq \nat\times \nat$ with $1\leq p_n\leq k_{n}$, $1\leq q_n \leq k_{-n}$ and $F=\{n_1<\ldots<n_\lambda\}\in [\nat]^{<\omega}_{>0}$ a finite non-empty subset of $\nat$ we set $y^{((p_n,q_n))_{n\in\nat}}_F = x_{w_{n_1}(p_{n_1},q_{n_1})\star \ldots \star w_{n_\lambda}(p_{n_\lambda},q_{n_\lambda})}$ (resp. $y^{(p_n)_{n\in\nat}}_F =x_{w_{n_1}(p_{n_1})\star \ldots \star w_{n_\lambda}(p_{n_\lambda})}$). Then
\begin{center}
$R_1$-$\lim_{w\in \widetilde{E}(\vec{w})}x_w=x_0$ if and only if $IP$-$\lim_{F\in [\nat]^{<\omega}_{>0}}y^{((p_n,q_n))_{n\in\nat}}_F =x_0$ uniformly
\end{center}
\begin{center}
 for all sequences $((p_n,q_n))_{n\in\nat}\subseteq \nat\times \nat$ with $1\leq p_n\leq k_{n}$, $1\leq q_n \leq k_{-n}$
\end{center}
\begin{center}
(resp. $R_2$-$\lim_{w\in E(\vec{w})}x_w=x_0$ if and only if $IP$-$\lim_{F\in [\nat]^{<\omega}_{>0}}y^{(p_n)_{n\in\nat}}_F =x_0$ uniformly
\end{center}
\begin{center}
for all sequences $(p_n)_{n\in\nat}\subseteq \nat$ with $1\leq p_n\leq k_{n}$).
\end{center}
\end{prop}
\begin{proof}
$(\Rightarrow)$ Let $V$ be a neighborhood of $x_0.$ There exists $n_0\equiv n_0(V)\in \nat$ such that $x_w\in V$ for every $w\in\widetilde{E}(\vec{w}) $ (resp. $w\in E(\vec{w})$) with $\min\{-\max dom^-(w),\min dom^+(w)\}\geq n_0$ (resp. with $\min dom(w)\geq n_0$). So, for $F\in [\nat]^{<\omega}_{>0}$ with  $n_0<\min F$ we have that $y^{((p_n,q_n))_{n\in\nat}}_F\in V$ (resp. $y^{(p_n)_{n\in\nat}}_F\in V$) for all sequences $((p_n,q_n))_{n\in\nat}\subseteq \nat\times \nat$ with $1\leq p_n\leq k_{n}$, $1\leq q_n \leq k_{-n}$ (resp. $(p_n)_{n\in\nat}\subseteq \nat$ with $1\leq p_n\leq k_{n}$).

$(\Leftarrow)$ Toward to a contradiction we suppose that there exists a neighborhood $V$ of $x_0$ such that for every $n\in \nat$ there exists $u_n=w_{m^n_1}(p_{m^n_1},q_{m^n_1})\star \ldots \star w_{m^n_\lambda}(p_{m^n_\lambda},q_{m^n_\lambda})\in \widetilde{E}(\vec{w})$ (resp. $u_n= w_{m^n_1}(p_{m^n_1})\star \ldots \star w_{m^n_\lambda}(p_{m^n_\lambda})\in E(\vec{w})$) with $\min\{-\max dom^-(u_n),$ $\min dom^+(u_n)\}$ $\geq n$ (resp. $\min dom(u_n)\geq n$) and $x_{u_n}\notin V.$ We can suppose that $u_n<_{\textsl{R}_1}u_{n+1}$ (resp. $u_n<_{\textsl{R}_2}u_{n+1}$) for every $n\in \nat.$ According to the hypothesis there exists $n_0\in \nat$ such that $y^{((p_n,q_n))_{n\in\nat}}_F\in V$ (resp. $y^{(p_n)_{n\in\nat}}_F\in V$) for all sequences $((p_n,q_n))_{n\in\nat}\subseteq \nat\times \nat$ with $1\leq p_n\leq k_{n}$, $1\leq q_n \leq k_{-n}$ (resp. $(p_n)_{n\in\nat}\subseteq \nat$ with $1\leq p_n\leq k_{n}$) and all $F\in [\nat]^{<\omega}_{>0}$ with $\min F\geq n_0$. Then $x_{u_{n_0}}\in V$, a contradiction.
\end{proof}

We will now give some applications of Theorem~\ref{3} to topological dynamical systems extending fundamental recurrence results of Birkhoff (\cite{Bi}) and Furstenberg-Weiss (\cite{FuW}, \cite{Fu}). Firstly, we will introduce the notions of  $\widetilde{L}(\Sigma,\vec{k})$-systems and $L(\Sigma,\vec{k})$-systems of continuous maps of a topological space into itself.

\begin{defn} Let $X$ be a topological space, $\Sigma=\{\alpha_n \;:\;n \in \mathbb{Z}^{\ast}\}$ be an alphabet and   $\vec{k}=(k_n)_{n\in\mathbb{Z}^\ast}\subseteq \nat$ such that
$(k_n)_{n\in\nat}$, $(k_{-n})_{n\in\nat}$ are increasing sequences. A family $\{T^w\}_{w\in \widetilde{L}(\Sigma,\vec{k})}$ (resp. $\{T^w\}_{w\in L(\Sigma,\vec{k})}$) of continuous functions of $X$ into itself is an \textbf{$\widetilde{L}(\Sigma,\vec{k})$-system}  (resp. an \textbf{$L(\Sigma,\vec{k})$-system}) of $X$ if $T^{w_1}T^{w_2}=T^{w_1\star w_2}$ for $w_1<_{\textsl{R}_1}w_2$ (resp. for $w_1<_{\textsl{R}_2}w_2$).
\end{defn}

\begin{exa}\label{exa2}
Let $X$ be a topological space.

(1) Let $T : X \rightarrow X$ be a continuous map. For an alphabet
$\Sigma=(m_n)_{n\in \nat}\subseteq \nat$, $\vec{k}=(k_n)_{n\in\nat} \subseteq \nat$ an increasing sequence and $(l_n)_{n\in \nat} \subseteq \nat$ we define for
every $w= w_{n_1} \ldots  w_{n_\lambda}\in L(\Sigma,\vec{k})$
\begin{center}
$T^w =T^{l_{n_1} w_{n_1} + \ldots + l_{n_\lambda} w_{n_\lambda}}$.
\end{center}
 Then $\{T^w\}_{w\in L(\Sigma,\vec{k})}$ is an $L(\Sigma,\vec{k})$-system of $X$.
\newline
Moreover, for a sequence $\{T_n\}_{n\in \nat}$ of continuous maps from $X$ into itself defining
\begin{center}
 $T^{w}=T^{l_{n_1} w_{n_1}}_{n_1}\ldots T^{l_{n_\lambda} w_{n_\lambda}}_{n_\lambda}.$
 \end{center}
we have another  $L(\Sigma,\vec{k})$-system of $X$.

 (2) For a given sequence $\{T_n\}_{n\in \zat^{\ast}}$ of continuous maps from $X$ into itself,  $\Sigma=(\alpha_n)_{n\in \mathbb{Z}^\ast}\subseteq \nat,$ $\vec{k}=(k_n)_{n\in\mathbb{Z}^\ast}\subseteq \nat$ such that
$(k_n)_{n\in\nat},$ $(k_{-n})_{n\in\nat}$ are increasing sequences and $(l_n)_{n\in \mathbb{Z}^\ast} \subseteq \nat$ we define for $w= w_{n_1} \ldots  w_{n_\lambda}\in \widetilde{L}(\Sigma,\vec{k})$
 \begin{center}
 $T^{w_{n_1}\ldots w_{n_\lambda}}=T^{l_{n_1} w_{n_1}}_{n_1}\ldots T^{l_{n_\lambda} w_{n_\lambda}}_{n_\lambda}.$
 \end{center}
 Then $\{T^w\}_{w\in \widetilde{L}(\Sigma,\vec{k})}$ is an $\widetilde{L}(\Sigma,\vec{k})$-system of $X$.
\newline
 In particular, if $T,S : X \rightarrow X$ are two continuous maps, then we can replace $T_n$ with $T^n$ and $T_{-n}$ with $S^n$ for every $n\in\nat$.
 \end{exa}

Via Theorem~\ref{3}, we will prove the existence of strongly recurrent points in a compact metric space $X$ for an $\widetilde{L}(\Sigma,\vec{k})$-system as well as for an $L(\Sigma,\vec{k})$-system of it. Moreover, we will point out the way to locate such points.
\begin{thm}\label{5}
Let $\{T^w\}_{w\in \widetilde{L}(\Sigma,\vec{k})}$ (resp. $\{T^w\}_{w\in L(\Sigma,\vec{k})}$) be an $\widetilde{L}(\Sigma,\vec{k})$-system (resp. $L(\Sigma,\vec{k})$-system) of a compact metric space $(X, d)$, $\vec{w}\in \widetilde{L}^{\infty} (\Sigma, \vec{k} ; \upsilon)$ (resp. $\vec{w}\in L^{\infty} (\Sigma, \vec{k} ; \upsilon)$) and $x\in X$. Then there exist an extraction $\vec{u}\prec\vec{w}$ of $\vec{w}$ and $x_0\in X$ such that
$$\text{$R_1$-}\lim_{w\in \widetilde{E}(\vec{u})}T^w (x)=x_0\;\;(\text{resp.\;\;$R_2$-}\lim_{w\in E(\vec{u})}T^w (x)=x_0).$$
Moreover, $x_0$ is \textbf{$\vec{w}$-recurrent point}, in the sense that
$$\text{$R_1$-}\lim_{w\in \widetilde{E}(\vec{u})}T^w (x_0)=x_0\;\;(\text{resp.\;\;$R_2$-}\lim_{w\in E(\vec{u})}T^w (x_0)=x_0).$$
\end{thm}

\begin{proof}
According to Theorem~\ref{3} there exist an extraction $\vec{u}$ of $\vec{w}$ and $x_0\in X$ such that
$\text{$R_1$-}\lim_{w\in \widetilde{E}(\vec{u})}T^w (x)=x_0\;\; (\text{resp.\;\;$R_2$-}\lim_{w\in E(\vec{u})}T^w (x)=x_0).$

Let $\epsilon>0$. There exists $n_0 \in\nat$ such that $d(T^w (x), x_0)< \frac{\epsilon}{2}$ for every $w\in \widetilde{E}(\vec{u})$ with $\min\{-\max dom^-(w),\min dom^+(w)\}\geq n_0$ (resp.  $w\in E(\vec{u})$ with $\min dom(w)\geq n_0$). Let $w\in \widetilde{E}(\vec{u})$ with $\min\{-\max dom^-(w),\min dom^+(w)\}\geq n_0$ (resp. $w\in E(\vec{u})$ with $\min dom(w)\geq n_0$). Then $d(T^w (x), x_0)< \frac{\epsilon}{2}$. Since $T^w$ is continuous, there exists ${\delta >0}$ such that if $d(z,x_0)<\delta$, then $d(T^w(z),T^w(x_0))<\frac{\epsilon}{2}$. Choose $w_1\in \widetilde{E}(\vec{u})$ (resp. $w_1\in E(\vec{u})$) such that $d(T^{w_1} (x), x_0)< \delta$ and $w<_{\textsl{R}_1} w_1$ (resp. $w<_{\textsl{R}_2} w_1$). Then $d(T^w(T^{w_1} (x)),T^w(x_0))= d(T^{w\star w_1} (x),T^w(x_0)) <\frac{\epsilon}{2}$. Since $d(T^{w\star w_1} (x), x_0) <\frac{\epsilon}{2}$ we have that $d(T^w(x_0),x_0)<\epsilon$.
\end{proof}
In the following corollaries we will describe some consequences of Theorem~\ref{5} for the simplest $\widetilde{L}(\Sigma,\vec{k})$-system  generated by a single transformation.

For a semigroup $(X,+)$ and $(x_n)_{n\in \nat}\subseteq X$ let
\begin{center}
$FS((x_n)_{n\in \nat}) = \{x_{n_1}+\ldots +x_{n_\lambda} : \lambda\in\nat, n_1<\ldots<n_\lambda \in\nat\}.$
\end{center}

\begin{cor}\label{6}
Let $(X,d)$ be a compact metric space, $T : X \rightarrow X$ a continuous map and $(m_n)_{n\in \nat},(r_n)_{n\in \nat}\subseteq \nat$ with $m_n<m_{n+1}, r_n<r_{n+1}$ for $n\in \nat$. Then, there exist $x_0 \in X$ and sequences $(\alpha_n)_{n\in \nat}\subseteq \nat $, $(\beta_n)_{n\in \nat}\subseteq FS((m_n)_{n\in \nat})$, $(\gamma_n)_{n\in \nat}\subseteq FS((r_n)_{n\in \nat})$ such that
$$\textit{IP-}\lim_{F\in [\nat]^{<\omega}_{>0}} T^{\sum_{n\in F}\alpha_n + p_n \beta_n + q_n\gamma_n}(x_0) =x_0,\;
(\textit{in particular},\;\lim_{n} T^{\alpha_n + p_n \beta_n + q_n\gamma_n}(x_0) =x_0)$$
uniformly for all sequences $((p_n,q_n))_{n\in\nat}\subseteq \nat\times \nat$ with $0\leq p_n\leq n$, $0\leq q_n \leq n$.
\end{cor}
\begin{proof}
Let $\Sigma=(\alpha_n)_{n\in \mathbb{Z}^\ast}\subseteq \nat$ with $\alpha_{-n} = \alpha_n =n$ for $n\in\nat$ and $\vec{k}=(k_n)_{n\in\mathbb{Z}^\ast}\subseteq \nat$ with $k_{-n}=k_n =n+1$ for $n\in\nat$. For $w= w_{n_1} \ldots  w_{n_\lambda}\in \widetilde{L}(\Sigma,\vec{k})$ we set $T^{w_{n_1}\ldots w_{n_\lambda}}=T^{-n_1 w_{n_1}}\ldots T^{-n_i w_{n_i}} T^{n_{i+1} w_{n_{i+1}}}\ldots T^{n_\lambda w_{n_\lambda}},$ where $n_i=\max dom^-(w),$ $n_{i+1}=\min dom^+(w).$ Then
  $\{T^w\}_{w\in \widetilde{L}(\Sigma,\vec{k})}$ is an $\widetilde{L}(\Sigma,\vec{k})$-system of $X$ (see Example~\ref{exa2}(2)). Let $\vec{w}=(w_n)_{n\in \nat}\in \widetilde{L}^{\infty} (\Sigma, \vec{k} ; \upsilon)$ with $w_n =w_{-r_n}w_{m_n}$ where $w_{-r_n}=w_{m_n}=\upsilon$. We apply Theorem~\ref{5}. So, there exist an extraction $\vec{u}=(u_n)_{n\in \nat}\in \widetilde{L}^{\infty} (\Sigma, \vec{k} ; \upsilon)$ of $\vec{w}$ and $x_0\in X$ such that
$R_1\text{-}\lim_{w\in \widetilde{E}(\vec{u})}T^w (x_0)=x_0.$\\ \noindent According to Proposition~\ref{4}, if $y^{((p_n,q_n))_{n\in\nat}}_F = T^{u_{n_1}(p_{n_1},q_{n_1})\star \ldots \star u_{n_\lambda}(p_{n_\lambda},q_{n_\lambda})}(x_0)$, then
\begin{center}
$IP$-$\lim_{F\in [\nat]^{<\omega}_{>0}}y^{((p_n,q_n))_{n\in\nat}}_F =x_0$
\end{center}
uniformly for all sequences $((p_n,q_n))_{n\in\nat}\subseteq \nat\times \nat$ with $1\leq p_n\leq n+1$, $1\leq q_n \leq n+1$.

Let $T^{u_{n}((p_{n},q_{n}))}= T^{\alpha_n + (p_n-1) \beta_n + (q_n-1)\gamma_n}$, where $\beta_n \in FS((m_n)_{n\in \nat})$ and $\gamma_n \in FS((r_n)_{n\in \nat})$. Then
$IP$-$\lim_{F\in [\nat]^{<\omega}_{>0}} T^{\sum_{n\in F}\alpha_n + p_n \beta_n + q_n\gamma_n}(x_0) =x_0$, (in particular, $\lim T^{\alpha_n + p_n \beta_n + q_n\gamma_n}(x_0) =x_0$) uniformly for all sequences $((p_n,q_n))_{n\in\nat}\subseteq \nat\times \nat$ with $0\leq p_n\leq n$, $0\leq q_n \leq n$.
\end{proof}

\begin{cor}\label{7}
Let $(X,d)$ be a compact metric space, $T : X \rightarrow X$ a continuous map and $(m_n)_{n\in \nat},(r_n)_{n\in \nat}\subseteq \nat$ with $m_n<m_{n+1}, r_n<r_{n+1}$ for all $n\in \nat$. Then, there exist $x_0 \in X$ and sequences $(\alpha_n)_{n\in \nat}\subseteq \nat $, $(\beta_n)_{n\in \nat}\subseteq FS((m_n)_{n\in \nat})$, $(\gamma_n)_{n\in \nat}\subseteq FS((r_n)_{n\in \nat})$ such that for every $\epsilon>0$ there exists $n_0\in\nat$ which satisfies
 \begin{center}
$d(T^{ p_n \beta_n + q_n\gamma_n}(T^{\alpha_n}(x_0)), T^{\alpha_n}(x_0))<\epsilon$
\end{center}
for every $n\geq n_0$ and $((p_n,q_n))_{n\in\nat}\subseteq \nat\times \nat$ with $0\leq p_n\leq n$, $0\leq q_n \leq n$.
\end{cor}
\begin{proof}
It follows from  Corollary~\ref{6}.
\end{proof}
We will define now the recurrent subsets of a compact metric space $X$ for an $\widetilde{L}(\Sigma,\vec{k})$-system as well as for an $L(\Sigma,\vec{k})$-system of it.
\begin{defn}
Let $\{T^w\}_{w\in \widetilde{L}(\Sigma,\vec{k})}$ (resp. $\{T^w\}_{w\in L(\Sigma,\vec{k})}$) be an $\widetilde{L}(\Sigma,\vec{k})$-system (resp. $L(\Sigma,\vec{k})$-system) of continuous maps of a compact metric space $(X, d)$ and $\vec{w}\in\widetilde{L}^{\infty}(\Sigma,\vec{k};\upsilon)$ (resp. $\vec{w}\in L^{\infty}(\Sigma,\vec{k};\upsilon)$).  A closed subset $A$ of $X$ is said to be \textbf{$\vec{w}$-recurrent set}  if for any $m\in \nat,$ $\varepsilon>0$ and any point $x\in A$ there exist $y\in A$ and
$u\in \widetilde{EV}(\vec{w})$  with $\min\{-\max dom^-(u),\min dom^+(u)\}>m$ (resp. $u\in EV(\vec{w})$ with $\min dom(u)\}>m$)  such that $d(T^{u(p,q)}(y),x)<\varepsilon$ for every $1\leq p,q\leq m.$
\end{defn}
In the following example we point out the way to locate recurrent subsets of a compact metric space $X$ for a given $\widetilde{L}(\Sigma,\vec{k})$-system as well as for a given $L(\Sigma,\vec{k})$-system of it.
\begin{exa}\label{exa9}
Let $(X, d)$ be a compact metric space and let $\textsl{F}(X)$ be the set of all nonempty closed subsets of $X$ endowed with the Hausdorff metric $\hat{d}$ (where $\hat{d}(A,B)=\max[\sup_{x\in A}d(x,B) ,\; \sup_{x\in B}d(x,A)]$). Then $(\textsl{F}(X), \hat{d})$ is also a compact metric space. Let $\{T^w\}_{w\in \widetilde{L}(\Sigma,\vec{k})}$ (resp. $\{T^w\}_{w\in L(\Sigma,\vec{k})}$) be an $\widetilde{L}(\Sigma,\vec{k})$-system (resp. $L(\Sigma,\vec{k})$-system) of continuous maps of $(X, d)$. We define $\hat{T}^w : \textsl{F}(X)\rightarrow \textsl{F}(X)$ with $\hat{T}^w (A) = T^w (A)$. Then $\{\hat{T}^w\}_{w\in \widetilde{L}(\Sigma,\vec{k})}$ (resp. $\{\hat{T}^w\}_{w\in L(\Sigma,\vec{k})}$) is an $\widetilde{L}(\Sigma,\vec{k})$-system (resp. $L(\Sigma,\vec{k})$-system) of  $(\textsl{F}(X), \hat{d})$. According to Theorem~\ref{5}, for every $\vec{w}=(w_n)_{n\in \nat}\in \widetilde{L}^{\infty} (\Sigma, \vec{k} ; \upsilon)$ (resp. $\vec{w}=(w_n)_{n\in \nat}\in L^{\infty} (\Sigma, \vec{k} ; \upsilon)$) there exist $A\in \textsl{F}(X)$ and an extraction $\vec{u}\prec\vec{w}$ of $\vec{w}$ such that
$$\text{$R_1$-}\lim_{w\in \widetilde{E}(\vec{u})}\hat{T}^w (A)=A\;\;(\text{resp.\;\;$R_2$-}\lim_{w\in E(\vec{u})}\hat{T}^w (A)=A).$$
Then $A$ is $\vec{w}$-recurrent in $(X, d)$. Observe that it is enough  $\text{$R_1$-}\lim_{w\in \widetilde{E}(\vec{u})}\hat{T}^w (A)\supseteq A$ $(\text{resp.\;\;$R_2$-}\lim_{w\in E(\vec{u})}\hat{T}^w (A)\supseteq A)$ in order $A$ to be $\vec{w}$-recurrent.
\end{exa}

\begin{prop}\label{12}
Let $A$ be a $\vec{w}$-recurrent subset of a compact metric space $(X,d).$ Then for every $\varepsilon>0$ and $m\in \nat$ there exist $u\in \widetilde{EV}(\vec{w})$ with $\min\{-\max dom^-(u) ,$ $ \min dom^+(u)\}>m$ (resp. $u\in EV(\vec{w})$ with $\min dom(u)>m$) and $z\in A$ such that
\begin{center}
$d(T^{u(p,q)}(z),z)<\varepsilon$ for every $1\leq p,q\leq m.$
\end{center}
\end{prop}

\begin{proof}
Let $\varepsilon>0$ and $m\in \nat$. For a $z_0\in A$
and $\varepsilon_1=\varepsilon/2$ there exist $z_1\in A$ and $u_1\in \widetilde{EV}(\vec{w})$ with $\min\{-\max dom^-(u_1),\min dom^+(u_1)\}$ $>m$ (resp. $u_1\in EV(\vec{w})$ with $\min dom(u)>m$) such that $d(T^{u_1(p,q)}z_1,z_0)<\varepsilon$ for every $1\leq p,q\leq m.$

Let have been chosen $z_0,z_1,\ldots,z_r\in A,$ $u_1<_{\textsl{R}_1}\ldots <_{\textsl{R}_1} u_r\in \widetilde{EV}(\vec{w})$ (resp. $u_1<_{\textsl{R}_2}\ldots <_{\textsl{R}_2} u_r\in EV(\vec{w})$) such that
$d(T^{u_{i}(p_{i},q_{i})\star\ldots\star u_j(p_j,q_j)}(z_j),z_{i-1})<\varepsilon/2$ for every $1\leq i\leq j\leq r$ and $1\leq p_l,q_l\leq m,$ for all $i\leq l\leq j.$

Since $T^w$ are continuous functions, there is $\varepsilon_r <\varepsilon/2$ such that if $d(z,z_r)<\varepsilon_r$ then $d(T^{u_{i}(p_{i},q_{i})\star\ldots\star u_r(p_r,q_r)}(z),z_{i-1})<\varepsilon/2$ for every $1\leq i\leq r$ and $1\leq p_l,q_l\leq m,$ for all $i\leq l\leq r.$ Since  $A$ is $\vec{w}$-recurrent, there exist $z_{r+1}\in A$ and $u_{r+1}\in \widetilde{EV}(\vec{w})$ with $u_r<_{\textsl{R}_1} u_{r+1}$ (resp. $u_{r+1}\in EV(\vec{w})$ with $u_r<_{\textsl{R}_2} u_{r+1}$) such that $d(T^{u_{r+1}(p,q)}(z_{r+1}),z_r)<\varepsilon_{r}$ for every $1\leq p,q\leq m.$ Hence, $d(T^{u_{i}(p_{i},q_{i})\star\ldots\star u_{r+1}(p_{r+1},q_{r+1})}(z_{r+1}),z_{i-1})<\varepsilon/2$ for every $1\leq i\leq r+1$ and $1\leq p_l,q_l\leq m,$ for all $i\leq l\leq r+1.$

Since $(X,d)$ is compact, there exist $i<j\in\nat$ such that $d(z_i,z_j)<\varepsilon/2$. Hence, for $u=u_{i+1}\star\ldots\star u_j\in \widetilde{EV}(\vec{w})$ (resp. $u=u_{i+1}\star\ldots\star u_j\in EV(\vec{w})$) we have $d(T^{u(p,q)}z_j,z_j)<\varepsilon$ for every $1\leq p,q\leq m.$
\end{proof}

\begin{defn}
A closed subset $A$ of a compact metric space $X$ is \textbf{homogeneous} with respect to a set of transformations $\{T_i\}$ acting on $X$ if there exists a group of homeomorphisms $G$ of $X$ each of which commutes with each $T_i$ and such that $G$ leaves $A$ invariant and $(A,G)$ is minimal (no proper closed subset of $A$ is invariant under the action of $G$).
\end{defn}
In the following proposition we give a sufficient condition in order a homogeneous subset to be strongly recurrent.
\begin{prop}\label{l1}
Let $A$ is a homogeneous set in a compact metric space $X$ with respect to the system $\{T^w\}_{w\in \widetilde{L}(\Sigma,\vec{k})}$ (resp. $\{T^w\}_{w\in L(\Sigma,\vec{k})}$) and $\vec{w}\in \widetilde{L}^{\infty}(\Sigma,\vec{k};\upsilon)$ (resp. $\vec{w}\in L^{\infty}(\Sigma,\vec{k};\upsilon)$). If for every $\varepsilon>0$ and $m\in \nat$ there exist $x,y\in A$ and $u\in \widetilde{EV}(\vec{w})$ with $\min\{-\max dom^-(u),\min dom^+(u)\}>m$ (resp. $u\in EV(\vec{w})$ with $\min dom(u)\}>m$) such that $d(T^{u(p,q)}(y),x)<\varepsilon$ for every $1\leq p,q\leq m,$ then $A$ is $\vec{w}$-recurrent.
\end{prop}
\begin{proof}
Let $\varepsilon>0$, $m\in \nat,$ and $G$ be a group of homeomorphisms commuting with $\{T^w\},$ and such that $G$ leaves $A$ invariant and $(A,G)$ is minimal. Let $\{U_1,\ldots,U_r\}$ be a finite covering of $A$ by open sets of diameter $<\varepsilon/2$. Then, from the minimality of $A$, we can find for each $1\leq i\leq r$ a finite set $\{g_{1}^{i},\ldots,g_{1}^{l_i}\}\subseteq G$ such that $\bigcup_{j=1}^{l_i} (g^{i}_{j})^{-1}(U_i)=A.$ Let $G_0=\{g^{i}_{j} : 1\leq i\leq r, 1\leq j\leq l_i\}\subseteq G$. Then for any $x,y\in A$ we have $\min_{g\in G_0}d(g(x),y)<\varepsilon/2.$

Let $\delta>0$ such that if $d(x_1,x_2)<\delta$, then $d(g(x_1),g(x_2))<\varepsilon$ for every $g\in G_0.$ According to the hypothesis, there exist $x,y\in A$ and $u\in \widetilde{EV} (\vec{w})$ with $\min\{-\max dom^-(u),$ $\min dom^+(u)\}>m$ (resp. $u\in EV(\vec{w})$ with $\min dom(u)>m$) such that ${d(T^{u(p,q)}(y),x)<\delta}$ for every $1\leq p,q\leq m$. Then

$d(T^{u(p,q)}(g(y)),g(x))= d(g(T^{u(p,q)}(y)),g(x))<\varepsilon/2$ for every $1\leq p,q\leq m.$

\noindent For a point $z\in A,$ find $g\in G_0$ with $d(g(x),z)<\varepsilon/2$. Then $d(T^{u(p,q)}(g(y)),z)\leq d(T^{u(p,q)}(g(y)),g(x))+ d(g(x),z)<\varepsilon$ for every $1\leq p,q\leq m.$ It follows that $A$ is $\vec{w}$-recurrent.
\end{proof}
We will prove now that a recurrent homogeneous subset $A$ of a compact metric space $X$ contains recurrent points, moreover these points consist a dense subset of $A.$

\begin{prop}\label{l3}
Let $\{T^w\}_{w\in \widetilde{L}(\Sigma,\vec{k})}$ (resp. $\{T^w\}_{w\in L(\Sigma,\vec{k})}$) be an $\widetilde{L}(\Sigma,\vec{k})$-system (resp. $L(\Sigma,\vec{k})$-system) of continuous transformations of a compact metric space $(X, d)$ and $\vec{w}\in\widetilde{L}^{\infty}(\Sigma,\vec{k};\upsilon)$ (resp. $\vec{w}\in L^{\infty}(\Sigma,\vec{k};\upsilon)$). A $\vec{w}$-recurrent homogeneous subset $A$ of $X$ contains $\vec{w}$-recurrent points ($x_0$ is $\vec{w}$-recurrent iff $\text{$R_1$-}\lim_{w\in \widetilde{E}(\vec{u})}T^w (x_0)=x_0$ (resp. iff $R_2$-$\lim_{w\in E(\vec{u})}T^w (x_0)=x_0$) for some $\vec{u}\prec \vec{w}$).

Moreover, the $\vec{w}$-recurrent points of $A$ consist a dense subset of $A$.
\end{prop}

\begin{proof}
Let $V$ be an open subset of $X$ such that $V\cap A\neq \emptyset$ and let $V'\subseteq V$ be an open set such that $V'\cap A\neq \emptyset$ and if $d(x,V')<\delta$ for $\delta>0$ then $x\in V.$ Since $A$ is homogeneous, there exists a group $G$ of homeomorphisms commuting with $\{T^w\}$ and such that $G$ leaves $A$ invariant and $(A,G)$ is minimal. From the minimality of $A$, there exists a finite subset $G_0\subseteq G$ such that $A\subseteq \bigcup_{g\in G_0}g^{-1}(V')$.

 Choose $\varepsilon>0$ such that whenever $x_1,x_2\in X$ and $d(x_1,x_2)<\varepsilon,$ then $d(g(x_1),g(x_2))<\delta$ for every $g\in G_0.$ Since $A$ is $\vec{w}$-recurrent, according to Proposition~\ref{12}, for $m\in \nat$ there exist $z\in A$ and $u\in \widetilde{EV}(\vec{w})$ with $\min\{-\max dom^-(u) , \min dom^+(u)\}$ $>m$ (resp. $u\in EV(\vec{w})$ with $\min dom(u)>m$) such that $d(T^{u(p,q)}(z),z)<\varepsilon$ for all $1\leq p,q\leq m.$

There exists $g\in G_0$ with $g(z)\in V'$ and since $d(T^{u(p,q)}(g(z)),g(z))<\delta$ for every $1\leq p,q\leq m,$ we have that $T^{u(p,q)}(g(z))\in V$ for every $1\leq p,q\leq m.$ Hence, each open set $V$ with $V\cap A\neq \emptyset$ contains a point $z'=g(z)\in A$ with $T^{u(p,q)}z'\in V$ for every $1\leq p,q\leq m.$ Since $T^w$ are continuous, we conclude that for every open set $V$ with $V\cap A\neq \emptyset$ and every $m\in \nat$ there exists an open set $V_1$ such that $\overline{V_{1}}\subseteq V$ and $T^{u(p,q)}V_1\subseteq V$ for every $1\leq p,q\leq m,$ for some $u\in \widetilde{EV}(\vec{w})$ with $\min\{-\max dom^-(u),\min dom^+(u)\}>m$ (resp. $u\in EV(\vec{w})$ with $\min dom(u)>m$).

Let $V_0$ be an open subset of $X$ such that $V_0 \cap A\neq \emptyset$. Inductively we can define a sequence $(V_n)_{n\in \nat}$ of open sets and a sequence $\vec{u}=(u_n)_{n\in \nat}\in \widetilde{L}^{\infty}(\Sigma,\vec{k};\upsilon)$ (resp.  $\vec{u}=(u_n)_{n\in \nat}\in L^{\infty}(\Sigma,\vec{k};\upsilon)$) with $\vec{u}\prec \vec{w}$ such that $\overline{V_{n}}\subseteq V_{n-1},$ $V_{n}\cap A\neq \emptyset$ and  $T^{u_{n}(p_{n},q_{n})}V_{n}\subseteq V_{n-1}$ for every $n\in\nat$ and $1\leq p_{n}\leq k_{n},$ $1\leq q_{n}\leq k_{-n}$.  We can also suppose that the diameter of $V_n$ tends to $0$. Then $\bigcap_{n\in\nat}V_{n}\cap A =\{x_0\}.$

For $1<i_1<\ldots<i_k,$ we have that $T^{u_{i_1}(p_{i_1},q_{i_1})\star\ldots\star u_{i_k}(p_{i_k},q_{i_k})}V_{i_k}\subseteq V_{i_1-1}.$ Then $T^w(x_0) \in V_i$ for every $w\in \widetilde{E}(\vec{u})$ with $u_{i+1}<_{\textsl{R}_1} w$ (resp. $w\in E(\vec{u})$ with $u_{i+1}<_{\textsl{R}_2} w$) so $R_1\text{-}\lim_{w\in \widetilde{E}(\vec{u})}T^w(x_0)=x_0$ (resp. $R_2$-$\lim_{w\in \widetilde{E}(\vec{u})}T^w (x_0)=x_0$). Hence, $x_0\in A\cap V_0$ is a  $\vec{w}$-recurrent point. This gives that the set of $\vec{w}$-recurrent points in $A$ is dense in $A$.
\end{proof}

Now, we shall prove a multiple recurrence theorem extending Theorem~\ref{5}, in case the transformations are homeomorphisms. We can say that the following theorem is the $``$word''-analogue of Birkhoff's multiple recurrence theorem.

\begin{thm}\label{114}
Let $\{T_1^w\}_{w\in \widetilde{L}(\Sigma,\vec{k})},\ldots,\{T_m^w\}_{w\in \widetilde{L}(\Sigma,\vec{k})}$ $($resp. $\{T_1^w\}_{w\in L(\Sigma,\vec{k})},\ldots,\{T_m^w\}_{w\in L(\Sigma,\vec{k})})$ be $m$ systems of transformations of a compact metric space $X,$ all contained in a commutative group $G$ of homeomorphisms of $X$ and let $\vec{w}\in\widetilde{L}^{\infty}(\Sigma,\vec{k};\upsilon)$ (resp. $\vec{w}\in L^{\infty}(\Sigma,\vec{k};\upsilon)$). Then, there exist $x_0\in X$ and an extraction $\vec{u}\prec \vec{w}$ such that
\begin{center}
$\text{$R_1$-}\lim_{w\in \widetilde{E}(\vec{u})}T_i^w (x_0)=x_0$ (resp. $\text{$R_2$-}\lim_{w\in E(\vec{u})}T_i^w (x_0)=x_0$) for every $1\leq i\leq m.$
\end{center}
Moreover, in case $(X,G)$ is minimal, the set of such points $x_0$ is a dense subset of $X$.
 \end{thm}

\begin{proof}
We assume without loss of generality that $(X,G)$ is minimal, otherwise we replace $X$ by a $G$-minimal subset of $X$. For $m=1$ we obtain the assertion from Theorem~\ref{5}. We proceed by induction. Suppose that the theorem holds for $m\in \nat$ and that $\{T_1^w\}_{w\in \widetilde{L}(\Sigma,\vec{k})},\ldots,\{T_{m+1}^w\}_{w\in \widetilde{L}(\Sigma,\vec{k})}$ (resp. $\{T_1^w\}_{w\in L(\Sigma,\vec{k})},\ldots,\{T_{m+1}^w\}_{w\in L(\Sigma,\vec{k})}$) are $m+1$ such  systems. We set $S_i^w=T_i^w(T_{m+1}^w)^{-1}$ for all $1\leq i\leq m.$ Then $S_i^{w_1\star w_2}=S_i^{w_1}S_i^{w_2}$ for every $1\leq i\leq m$ and $w_1<_{\textsl{R}_1}w_2$ (resp. $w_1<_{\textsl{R}_2}w_2$), since all the maps commute. By the induction hypothesis there exist $y\in X$ and $\vec{u}\prec \vec{w}$ such that $R_1$-$\lim_{w\in \widetilde{E}(\vec{u})}S_i^w (y)=y$ (resp. $R_2$-$\lim_{w\in E(\vec{u})}S_i^w (y)=y$) for every $1\leq i\leq m.$

Consider the product $X^{m+1}$ and let $\Delta^{m+1}$ be the diagonal subset consisting of the $(m+1)$-tuples $(x,\ldots,x)\in X^{m+1}.$ Identifying each $g\in G$ with $g\times\ldots\times g$ we can assume that $G$ acts on $X^{m+1}$. Also, the functions $T_1^w\times\ldots\times T_{m+1}^w$ acts on $X^{m+1}$ and commute with the functions of $G$. Since $G$ leaves $\Delta^{m+1}$ invariant and $(\Delta^{m+1},G)$ is minimal, $\Delta^{m+1}$ is a homogeneous set. According to Proposition~\ref{l3}, it suffices to prove that $\Delta^{m+1}$ is $\vec{w}$-recurrent. But, according to Proposition~\ref{l1}, the set $\Delta^{m+1}$ is $\vec{w}$-recurrent, since $\text{$R_1$-}\lim_{w\in \widetilde{E}(\vec{u})}(T_1^w\times\ldots\times T_{m+1}^w) [((T_{m+1}^w)^{-1}\times\ldots\times (T_{m+1}^w)^{-1})((y,\ldots,y))]=(y,\ldots,y)$ (resp. $\text{$R_2$-}\lim_{w\in E(\vec{u})}(T_1^w\times\ldots\times T_{m+1}^w) [((T_{m+1}^w)^{-1}\times\ldots\times (T_{m+1}^w)^{-1})((y,\ldots,y))]=(y,\ldots,y)$).
\end{proof}
Theorem~\ref{114} has the following consequence.

\begin{prop}\label{15}
Let $\{T_1^w\}_{w\in \widetilde{L}(\Sigma,\vec{k})},\ldots,\{T_m^w\}_{w\in \widetilde{L}(\Sigma,\vec{k})}$ $($resp. $\{T_1^w\}_{w\in L(\Sigma,\vec{k})},  \ldots,$ $ \{T_m^w\}_{w\in L(\Sigma,\vec{k})})$ be $m$ systems of transformations of a compact metric space $X,$ all contained in a commutative group $G$ of homeomorphisms of $X$, which acts minimally on $X$. For $\vec{w}\in\widetilde{L}^{\infty}(\Sigma,\vec{k};\upsilon)$ (resp. $\vec{w}\in L^{\infty}(\Sigma,\vec{k};\upsilon)$) and $U$ a non-empty open subset of $X$, there exists $\vec{u}\prec \vec{w}$ so that $$\bigcap^m_{i=1}(T_i^w)^{-1}(U)\neq \emptyset\;\;\text{for every}\;\;w\in \widetilde{E}(\vec{u})\;\;\;(\text{resp.}\;w\in E(\vec{u})).$$
\end{prop}

\begin{proof}
Since $G$ acts minimally on $X$, $X=\bigcup_{g\in G_0}g^{-1}(U)$, where $G_0$ is a finite subset of $G$. Let $\delta>0$ be such that every set of diameter $<\delta$ is contained in some $g^{-1}(U)$ for $g\in G_0.$ According to Theorem~\ref{114}, there exist $x_0\in X$ and $\vec{u}\prec \vec{w}$ such that $R_1\text{-}\lim_{w\in \widetilde{E}(\vec{u})}T_i^w (x_0)=x_0$ (resp. $R_2$-$\lim_{w\in E(\vec{u})}T_i^w (x_0)=x_0$) for every $1\leq i\leq m$. Refine $\vec{u}$ such that $d(T_i^w (x_0) ,x_0)<\delta/2$ for every $w\in \widetilde{E}(\vec{u})$ (resp. $w\in E(\vec{u})$) and $1\leq i\leq m.$ Then there exists $g\in G_0$ such that $T_i^w (x_0)\in g^{-1}(U)$ for every $w\in \widetilde{E}(\vec{u})$ (resp. $w\in E(\vec{u}))$ and $1\leq i\leq m$. Consequently, $g(x_0)\in \bigcap^m_{i=1}(T_i^w)^{-1}(U)$ for every $w\in \widetilde{E}(\vec{u})$ (resp. $w\in E(\vec{u})).$
\end{proof}

\section{Applications}

We will indicate the way in which the recurrence results for topological systems or nets indexed by words, that we proved in the previous section, can be applied to systems or nets indexed by semigroups that can be represented as words (Example~\ref{exa1}) and consequently to systems or nets indexed by an arbitrary semigroup.

\subsection*{Semigroup $(\qat,+)$}

As we described in Example~\ref{exa1}(1), the set $\mathbb{Q}^\ast$ of the nonzero rational numbers can be identified with a set $\widetilde{L}(\Sigma,\vec{k})$ of $\omega$-$\mathbb{Z}^\ast$-located words, via the function
\begin{center}
$g : \widetilde{L}(\Sigma,\vec{k}) \rightarrow \mathbb{Q}^\ast$, with   $\;g(q_{t_1}\ldots q_{t_l})=\sum_{t\in dom^-(w)}q_t\frac{(-1)^{-t}}{(-t+1)!}\;+\;\sum_{t\in dom^+(w)}q_t(-1)^{t+1}t!.$  \end{center}
We extend the function $g$ to the set $\widetilde{L}(\Sigma,\vec{k};\upsilon)$ of variable words corresponding to each $w=q_{t_1}\ldots q_{t_l}\in \widetilde{L}(\Sigma,\vec{k};\upsilon)$ a function $q=g(w)$ which sends every $(i,j)\in \nat\times\nat$ with $1\leq i\leq -\max dom^-(w),$ $1\leq j\leq \min dom^+(w),$ to $$q(i,j)=g(T_{(j,i)}(w))= \sum_{t\in C^-}q_t\frac{(-1)^{-t}}{(-t+1)!} + i\sum_{t\in V^-}\frac{(-1)^{-t}}{(-t+1)!} +  \sum_{t\in C^+}q_t(-1)^{t+1}t!+j\sum_{t\in V^+}(-1)^{t+1}t!,$$ where $C^-=\{t\in dom^-(w):\;q_t\in \Sigma\},$ $V^-=\{t\in dom^-(w):\;q_t=\upsilon\}$ and $C^+=\{t\in dom^+(w):\;q_t\in \Sigma\},$ $V^+=\{t\in dom^+(w):\;q_t=\upsilon\}.$ Let $\mathbb{Q}(\upsilon) = g(\widetilde{L}(\Sigma,\vec{k};\upsilon))$. Then the extended  function
\begin{center}
$g:\widetilde{L}(\Sigma\cup\{\upsilon\},\vec{k})\rightarrow \mathbb{Q}^\ast\cup \mathbb{Q}(\upsilon)$
\end{center}
is one-to-one and onto. For $q_1,q_2\in \mathbb{Q}^\ast\cup \mathbb{Q}(\upsilon)$ we define the relation
\begin{center} $q_1<_{\textsl{R}_1} q_2 \;\Longleftrightarrow\; g^{-1}(q_1)<_{\textsl{R}_1}g^{-1}(q_2).$
\end{center}
So, $\{x_q\}_{q\in \mathbb{Q}^\ast}\subseteq X$, where $X$ is a topological space, can be considered as an $ \textsl{R}_1$-net and consequently we can define, for $x_0\in X$,
$R_1$-$\lim_{q\in \mathbb{Q}^\ast} x_q=x_0$ iff for any neighborhood $V$ of $x_0,$ there exists $n_0\equiv n_0(V)\in \nat$ such that $x_q\in V$ for every $q\in \mathbb{Q}^\ast$ with $\min\{-\max dom^-(g^{-1}(q)),\min dom^+(g^{-1}(q))\}\geq n_0$.

Observe that $g(w_1\star w_2)=g(w_1)+g(w_2)$ for every $w_1<_{\textsl{R}_1}w_2\in \widetilde{L}(\Sigma\cup\{\upsilon\},\vec{k}).$ So, if

$\vec{q}=(q_n)_{n\in \nat}\in\mathbb{Q}^{\infty}(\upsilon)=\{(q_n)_{n\in\nat} : q_n\in \mathbb{Q}(\upsilon)$
and  $q_n<_{\textsl{R}_1}q_{n+1}\}$,
\newline
then the set of the extractions of $\vec{q}$ is

$\widetilde{EV}^{\infty}(\vec{q}) = \{\vec{r}=(r_n)_{n\in\nat} \in \mathbb{Q}^{\infty}(\upsilon) : r_n=g(u_n)$ for $(u_n)_{n\in\nat}\in \widetilde{EV}^{\infty}((g^{-1}(q_n))_{n\in\nat})\}$ and
\newline
the set of all the extracted rationals of $\vec{q}$ is

$\widetilde{E}(\vec{q})=\{q\in FS[(q_n(i_n,j_n))_{n\in \nat}] : ((i_n,j_n))_{n\in \nat}\subseteq \nat\times\nat$ with $1\leq i_n,j_n\leq n\} =$
\begin{center} $=\{g(w) : w\in \widetilde{E}((g^{-1}(q_n))_{n\in\nat})\}.\;\;\;\;\;\;\;\;\;\;\; \;\;\;\;\;\;\;\;\;\;\;\; \;\;\;\;\;\;\;\;\;\;\;\;\;\;\;\;\;\;\;\;\;\;\;\; \;\;\;\;\;\;\;\;\;\;\;\; \;\;$\end{center}

\noindent Of course, $\{x_q\}_{q\in \widetilde{E}(\vec{q})}$ is an $ \textsl{R}_1$-subnet of $\{x_q\}_{q\in \mathbb{Q}^\ast}$.

Hence, via the function $g$, all the presented results relating to $\omega$-$\mathbb{Z}^\ast$-located words give analogous results for the rational numbers. For example Theorems~\ref{3},~\ref{114} give the following:

\begin{thm}\label{16}
For every net $\{x_ q\}_{ q\in \mathbb{Q}^\ast}$ in a compact metric space $(X,d)$ and $\vec{q}=(q_n)_{n\in\nat} \in \mathbb{Q}^{\infty}(\upsilon)$ there exist an extraction $\vec{r}=(r_n)_{n\in\nat}$ of $\vec{q}$ and $x_0\in X$ such that
\begin{center}
$\text{$R_1$-}\lim_{q\in FS[(r_n(i_n,j_n))_{n\in \nat}]}x_q=x_0$ (in particular $x_{r_n(i_n,j_n)}\rightarrow x_0$),
\end{center}
uniformly for every $((i_n,j_n))_{n\in \nat}\subseteq \nat\times\nat$ with $1\leq i_n,j_n\leq n$.
\end{thm}
We call a family $\{T^q\}_{q\in {\mathbb{Q}^\ast}}$ of continuous functions of a topological space $X$ into itself a  $\mathbb{Q}^\ast$-system of $X$ if $T^{q_1}T^{q_2}=T^{q_1 + q_2}$ for $q_1<_{\textsl{R}_1}q_2$.

\begin{thm}\label{17}
Let $\{T_1^q\}_{q\in {\mathbb{Q}^\ast}},\ldots,\{T_m^q\}_{q\in {\mathbb{Q}^\ast}}$ be $m$ $\mathbb{Q}^\ast$-systems of transformations of a compact metric space $X,$ all contained in a commutative group $G$ of homeomorphisms of $X$ and let $\vec{q}\in\mathbb{Q}^{\infty}(\upsilon)$. Then, there exist $x_0\in X$ and an extraction $\vec{r}=(r_n)_{n\in\nat}$ of $\vec{q}$ such that, for every $1\leq i\leq m$,
\begin{center}
$\text{$R_1$-}\lim_{q\in FS[(r_n(i_n,j_n))_{n\in \nat}]}T_i^q (x_0)=x_0$ (in particular, $T_i^{r_n(i_n,j_n)}(x_0)\rightarrow x_0$),
\end{center}
  uniformly for all $((i_n,j_n))_{n\in \nat}\subseteq \nat\times\nat$ with $1\leq i_n,j_n\leq n$.

Moreover, in case $(X,G)$ is minimal, the set of such points $x_0$ is a dense subset of $X.$
\end{thm}

\subsection*{Semigroup $(\zat,+)$}

As we described in Example~\ref{exa1}(2), for a given increasing sequence $(k_n)_{n\in \nat}\subseteq \nat$ with $k_n \geq 2$, the set $\zat^\ast$ of the nonzero integer numbers can be identified with a set $L(\Sigma,\vec{k})$ of $\omega$-located words, via the function
\begin{center}
$g : L(\Sigma,\vec{k}) \rightarrow \zat^\ast$, with $\;g(z_{s_1}\ldots z_{s_l})=\sum^{l}_{i=1 }z_{s_i}(-1)^{s_i-1}l_{s_i-1}$
\end{center}
where $l_0=1$ and $l_s=k_1\ldots k_s ,$ for $s>0$.

We extend the function $g$ to the set $L(\Sigma,\vec{k};\upsilon)$ of variable $\omega$-located words corresponding to each $w=z_{s_1}\ldots z_{s_l}\in L(\Sigma,\vec{k};\upsilon)$ a function $z=g(w)$ which sends every $i\in \nat$ with $1\leq i\leq k_{\min dom(w)},$ to
\begin{center}
$z(i)=g(T_{i}(w))=\sum_{s\in C}z_s(-1)^{s-1}l_{s-1} + \sum_{s\in V} i (-1)^{s-1}l_{s-1}$.
\end{center}
where $C=\{s\in dom(w): z_s\in \Sigma\}$ and $V=\{s\in dom(w):\;z_s=\upsilon\}$.

Let $\zat(\upsilon) = g(L(\Sigma,\vec{k};\upsilon))$. Then the extended  function
\begin{center}
$g:L(\Sigma\cup\{\upsilon\},\vec{k})\rightarrow \zat^\ast\cup \zat(\upsilon)$
\end{center}
is one-to-one and onto. For $z_1,z_2\in \zat^\ast\cup \zat(\upsilon)$ we define the relation
\begin{center} $z_1 <_{\textsl{R}_2} z_2 \;\Longleftrightarrow\; g^{-1}(z_1)<_{\textsl{R}_2}g^{-1}(z_2).$
\end{center}
So, $\{x_z\}_{z\in \zat^\ast}\subseteq X$, where $X$ is a topological space, can be considered as an $ \textsl{R}_2$-net and consequently we can define, for $x_0\in X$,
$R_2$-$\lim_{z\in \zat^\ast} x_z=x_0$ iff for any neighborhood $V$ of $x_0,$ there exists $n_0\equiv n_0(V)\in \nat$ such that $x_z\in V$ for every $z\in \zat^\ast$ with $\min dom(g^{-1}(z))\geq n_0$.

Observe that $g(w_1\star w_2)=g(w_1)+g(w_2)$ for every $w_1<_{\textsl{R}_2}w_2\in L(\Sigma\cup\{\upsilon\},\vec{k}).$ So, if

$\vec{z}=(z_n)_{n\in \nat}\in\zat^{\infty}(\upsilon)=\{(z_n)_{n\in\nat} : z_n\in \zat(\upsilon)$
and  $z_n<_{\textsl{R}_2}z_{n+1}\}$,
\newline
then the set of the extractions of $\vec{z}$ is

$EV^{\infty}(\vec{z}) = \{\vec{v}=(v_n)_{n\in\nat} \in \zat^{\infty}(\upsilon) : v_n=g(u_n)$ for $(u_n)_{n\in\nat}\in EV^{\infty}((g^{-1}(z_n))_{n\in\nat})\}$ and
\newline
the set of all the extracted integers of $\vec{z}$ is

$E(\vec{z})=\{z\in FS[(z_n(i_n))_{n\in \nat}] : (i_n)_{n\in \nat}\subseteq \nat$ with $1\leq i_n\leq k_n\} =$

\begin{center} $=\{g(w) : w\in E((g^{-1}(z_n))_{n\in\nat})\}.\;\;\;\;\;\;\;\;\;\;\;\;\;\;\;\;\;\;\;\; \;\;\;\;\;\;\;\;\;\;\;\;\;\;\;\;\;\;\;\;\;\;\;\;\;\;\;\;\;\;\;\;\;\;\;\;\;\;\;\;\;$\end{center}

\noindent Of course, $\{x_z\}_{z\in E(\vec{z})}$ is an $ \textsl{R}_2$-subnet of $\{x_z\}_{z\in \zat^\ast}$.

Hence, via the function $g$, all the presented results relating to $\omega$-located words give analogous results for the integers. For example Theorems~\ref{3},~\ref{5} give the following.

\begin{thm}\label{18}
For every net $\{x_ z\}_{ z\in \zat^\ast}$ in a compact metric space $(X,d)$,
and $\vec{z}=(z_n)_{n\in\nat} \in \zat^{\infty}(\upsilon)$ there exist an extraction $\vec{v}=(v_n)_{n\in\nat}$ of $\vec{z}$ and $x_0\in X$ such that
\begin{center}
$\text{$R_2$-}\lim_{z\in FS[(v_n(i_n))_{n\in \nat}]}x_z=x_0$ (in particular $x_{v_n(i_n)}\rightarrow x_0$),
\end{center}
uniformly for all $(i_n)_{n\in \nat}\subseteq \nat$ with $1\leq i_n\leq k_n$.
\end{thm}
We call a family $\{T^z\}_{z\in \zat^\ast}$ of continuous functions of a topological space $X$ into itself a  $\zat^\ast$-system of $X$ if $T^{z_1}T^{z_2}=T^{z_1 + z_2}$ for $z_1<_{\textsl{R}_2}z_2$.

\begin{thm}\label{19}
Let $\{T^z\}_{z\in \zat^\ast}$ be a $\zat^\ast$-system of continuous maps of a compact metric space $(X, d)$, $\vec{z}=(z_n)_{n\in \nat}\in \zat^{\infty}(\upsilon)$ and $y\in X$. Then there exist an extraction $\vec{v}=(v_n)_{n\in\nat}$ of $\vec{z}$ and $x_0\in X$ such that
\begin{center}
$\text{$R_2$-}\lim_{z\in FS[(v_n(i_n))_{n\in \nat}]} T^z (y)=x_0,$
$\text{$R_2$-}\lim_{z\in FS[(v_n(i_n))_{n\in \nat}]}T^z (x_0)=x_0$
\end{center}
uniformly for all $(i_n)_{n\in \nat}\subseteq \nat$ with $1\leq i_n\leq k_n$.
\end{thm}
As we described in Example~\ref{exa1}(3), the set of natural numbers can be identified with a set $L(\Sigma,\vec{k})$ and consequently all the presented results relating to $\omega$-located words give analogous recurrence results for the natural numbers.

We will now give some applications of the previously mentioned recurrence results for systems or nets indexed by words to systems or nets indexed by an arbitrary semigroup. For simplicity we will present only the case of commutative semigroups.

Let $(S,+)$ be a semigroup and $(y_{l,n})_{n \in \zat^\ast}\subseteq S$ for every $l\in \zat^\ast$. Setting $\Sigma=\{\alpha_n:n\in\mathbb{Z}^\ast\}$, where $\alpha_n=n$ for $n\in \zat^\ast$  and $\vec{k}=(k_n)_{n \in \mathbb{Z}^\ast}\subseteq\nat$, where $(k_n)_{n \in \nat}$ and $(k_{-n})_{n \in \nat}$ are increasing sequences, we define the function \begin{center}$\varphi:\widetilde{L}(\Sigma,\vec{k})\rightarrow S \;\;\text{with}\;\; \varphi(w_{n_1}\ldots w_{n_m})=\sum^{m}_{i=1}y_{w_{n_i},n_i}.$
\end{center}
We extend the function $\varphi$ to the set $\widetilde{L}(\Sigma,\vec{k};\upsilon)$ of variable words corresponding to each $w=w_{n_1}\ldots w_{n_m}\in \widetilde{L}(\Sigma,\vec{k};\upsilon)$ a function $s=\varphi(w)$ which sends every $(i,j)\in \nat\times\nat$ with $1\leq j\leq -\max dom^-(w),$ $1\leq i\leq \min dom^+(w),$ to $s(i,j)=\varphi(T_{(i,j)}(w))\in S$. In case $(S,+)$ is a commutative semigroup
\begin{center}
$s(i,j)=\varphi(w)((i,j))=\sum_{t \in C}y_{w_{t},t}+\sum_{t \in V^+}y_{i,t}+\sum_{t \in V^-}y_{-j,t},$,
\end{center}
where $C=\{n\in dom(w):w_n\in \Sigma\},$ $V^-=\{n\in dom^-(w): w_n=\upsilon\}$ and $V^+=\{n\in dom^+(w): w_n=\upsilon\}.$

For a subset $\{x_s : s\in S\}$ of a topological space $X$ we can consider the $ \textsl{R}_1$-net  $\{x_{\varphi(w)}\}_{w\in \widetilde{L}(\Sigma,\vec{k})}$ in $X$. Let $\vec{w}=(w_n)_{n\in\nat}\in \widetilde{L}^{\infty} (\Sigma, \vec{k} ; \upsilon)$ such that $R_1$-$\lim_{w\in \widetilde{E}(\vec{w})}x_{\varphi(w)}=x_0$, for $x_0\in X$. Then setting, for every $n\in\nat$,
\begin{center}
$s_n=\varphi(w_n) : \{1,\ldots,k_n\}\times \{1,\ldots,k_{-n}\}\rightarrow X $ with $s_n(i,j)=\sum_{t \in C_n}y_{w_{t},t}+\sum_{t \in V^+_n}y_{i,t}+\sum_{t \in V^-_n}y_{-j,t},$
\end{center}
we have that
\begin{center}
 $R_1$-$\lim_{s\in FS[ (s_n(i_n,j_n))_{n\in \nat}]}x_{s}=x_0$
\end{center}
uniformly for all $((i_n,j_n))_{n\in \nat}\subseteq \nat\times\nat$ with $1\leq i_n\leq k_n$, $1\leq j_n\leq k_{-n}$. We write $R_1$-$\lim_{s\in FS[ (s_n(i_n,j_n))_{n\in \nat}]}x_{s}=x_0$ if and only if for any neighborhood $V$ of $x_0,$ there exists $n_0\equiv n_0(V)\in \nat$ such that $x_s\in V$ for every $s\in FS\big[ \big(s_n(i_n,j_n)\big)_{n\geq n_0}\big]$.

Hence, via the function $\varphi$, all the presented results related to $\omega$-$\mathbb{Z}^\ast$-located words give analogous results for nets indexed by an arbitrary semigroup. For example Theorems~\ref{3},~\ref{114} give the following.

\begin{thm}\label{20}
Let $(S,+)$ be a commutative semigroup and $(y_{l,n})_{n \in \zat^\ast}\subseteq S$ for every $l\in \zat^\ast$. For every subset $\{x_s : s\in S\}$ of a compact metric space $(X,d)$ there exist $x_0\in X$ and, for every $n\in\nat$, functions
$s_n : \{1,\ldots,k_n\}\times \{1,\ldots,k_{-n}\}\rightarrow X$ with \begin{center} $s_n(i,j)=\sum_{t \in C_n}y_{w_{t},t}+\sum_{t \in V^+_n}y_{i,t}+\sum_{t \in V^-_n}y_{-j,t},$\end{center}
where $C_n=C_n^-\cup C_n^+\subseteq \zat^\ast$ with $\max C_{n+1}^-<\min C_{n}^-<\max C_n^+<\min C^+_{n+1},$ $ V^+_n\subseteq \nat$ with $\max V^+_n<\min V^+_{n+1}$ and $ V^-_n\subseteq \zat^-$ with $\min V^-_n >\max V^-_{n+1}$, such that
\begin{center}
$R_1$-$\lim_{s\in FS[ (s_n(i_n,j_n))_{n\in \nat}]}x_{s}=x_0$ (in particular $x_{s_n(i_n,j_n)}\rightarrow x_0$),
\end{center}
uniformly for every $((i_n,j_n))_{n\in \nat}\subseteq \nat\times\nat$ with $1\leq i_n\leq k_n,$ $1\leq j_n\leq k_{-n}.$
\end{thm}

\begin{cor}\label{cor: 21} Let $(S,+)$ be a commutative semigroup and $(y_{n})_{n \in \zat^\ast}\subseteq S$. For every subset $\{x_s : s\in S\}$ of a compact metric space $(X,d)$ and functions $p,q : \nat \rightarrow \nat$ there exist $x_0\in X$ and  $(a_n)_{n\in \nat}\subseteq FS[(y_n)_{n\in \mathbb{Z}^\ast}],$ $(b_n)_{n\in \nat}\subseteq FS[(y_n)_{n\in \nat}]$ and $(c_n)_{n\in \nat}\subseteq FS[(y_{-n})_{n\in \nat}]$ such that
\begin{center}
$R_1$-$\lim_{s\in FS[ (a_n + p(i_n) b_n + q(j_n) c_n)_{n\in \nat}]}x_{s}=x_0$ (in particular $x_{a_n + p(i_n) b_n + q(j_n) c_n}\rightarrow x_0$)
\end{center}
uniformly for every $((i_n,j_n))_{n\in \nat}\subseteq \nat\times\nat$ with $1\leq i_n,j_n\leq n$.
\end{cor}
\begin{proof} Set $y_{l,n}=p(l) y_n$ for every $l\in\nat$ and $y_{l,n}=q(-l)y_n$ for every $l\in \zat^-$ and apply Theorem~\ref{20}.
\end{proof}

Let $(S,+)$ be a commutative semigroup and $(y_{l,n})_{n \in \zat^\ast}\subseteq S$ for every $l\in \zat^\ast.$ We call a family $\{T^s\}_{s\in S}$ of continuous functions of a topological space $X$ into itself an  $\widetilde{L}(\Sigma,\vec{k})$-system of $S$ if $T^{\varphi(w_1)}T^{\varphi(w_2)}=T^{\varphi(w_1\star w_2)}$ for $w_1<_{\textsl{R}_1}w_2 \in \widetilde{L}(\Sigma,\vec{k})$.

\begin{thm}\label{22}
Let $(S,+)$ be a commutative semigroup, $(y_{l,n})_{n \in \zat^\ast}\subseteq S$ for every $l\in \zat^\ast$ and $\{T_1^s\}_{s\in S},\ldots,$ $\{T_m^s\}_{s\in S}$ be $m$ $\widetilde{L}(\Sigma,\vec{k})$-systems of transformations of a compact metric space $X,$ all contained in a commutative group $G$ of homeomorphisms of $X$. Then, there exist $x_0\in X$ and, for every $n\in\nat$, functions
$s_n : \{1,\ldots,k_n\}\times \{1,\ldots,k_{-n}\}\rightarrow X $ with \begin{center} $s_n(i,j)=\sum_{t \in C_n}y_{w_{t},t}+\sum_{t \in V^+_n}y_{i,t}+\sum_{t \in V^-_n}y_{-j,t},$
\end{center}
where $C_n=C_n^-\cup C_n^+\subseteq \zat^\ast$ with $\max C_{n+1}^-<\min C_{n}^-<\max C_n^+<\min C^+_{n+1},$ $ V^+_n\subseteq \nat$ with $\max V^+_n<\min V^+_{n+1}$ and $ V^-_n\subseteq \zat^-$ with $\min V^-_n >\max V^-_{n+1}$, such that
\begin{center}
$\text{$R_1$-}\lim_{s\in FS[ (s_n(i_n,j_n))_{n\in \nat}]}T_i^s (x_0)=x_0$  for every $1\leq i\leq m,$
\end{center}
uniformly for every $((i_n,j_n))_{n\in \nat}\subseteq \nat\times\nat$ with $1\leq i_n\leq k_n,$ $1\leq j_n\leq n.$

Moreover, in case $(X,G)$ is minimal, the set of such points $x_0$ is a dense subset of $X$.
\end{thm}

\medskip

{\bf{Acknowledgments.}} We thank Professor S. Negrepontis for helpful discussions and support during the preparation of this paper. The first author acknowledge partial support from the Kapodistrias research grant of Athens University. The second author acknowledge partial support from the State Scholarship Foundation of Greece.

\bigskip
{\footnotesize
\noindent
\newline
Vassiliki Farmaki:
\newline
{\sc Department of Mathematics, Athens University, Panepistemiopolis, 15784 Athens, Greece}
\newline
E-mail address: vfarmaki@math.uoa.gr

\medskip

\noindent
Andreas Koutsogiannis:
\newline
{\sc Department of Mathematics, Athens University, Panepistemiopolis, 15784 Athens, Greece}
\newline
E-mail address: akoutsos@math.uoa.gr

\end{document}